%% file: cmp_main1.tex
\documentclass[12pt]{amsart}
\usepackage{epsfig,color}
\usepackage[matrix,arrow]{xy}

\makeatother

\theoremstyle{definition}

\def\fnum{equation}
\newtheorem{Thm}[\fnum]{Theorem}
\newtheorem{Cor}[\fnum]{Corollary}

\newtheorem{Lem}[\fnum]{Lemma}
\newtheorem{Con}[\fnum]{Conjecture}
\newtheorem{Def}[\fnum]{Definition}

\newtheorem{Rem}[\fnum]{Remark}
\newtheorem{Pro}[\fnum]{Proposition}
\newtheorem{Obs}{Observation}

\numberwithin{equation}{section}

\newcommand{\Vol}{{\text{Vol}}}

\newcommand{\nn}{{\bf{n}}}

\newcommand{\Hess}{{\text {Hess}}}
\def\ZZ{{\bold Z}}
\def\RR{{\bold R}}

\def\SS{{\bold S}}

\newcommand{\e}{{\text {e}}}

\newcommand{\cO}{{\mathcal{O}}}

\newcommand{\cc}{\mathcal C}
\newcommand{\cs}{\mathcal S}

\DeclareMathOperator{\wh}{Wh}
\DeclareMathOperator{\Top}{Top}
\DeclareMathOperator{\pl}{PL}

\newcommand{\eqr}[1]{(\ref{#1})}

\title[Mean curvature flow as a tool to study topology of $4$-manifolds]{Mean curvature flow as a tool to study topology of $4$-manifolds}

\author{Tobias Holck Colding}%
\address{MIT, Dept. of Math.\\
77 Massachusetts Avenue, Cambridge, MA 02139-4307.}
\author{William P. Minicozzi II}%
\address{Johns Hopkins University\\
Dept. of Math.\\
3400 N. Charles St.\\
Baltimore, MD 21218.}
\author{Erik Kj\ae r Pedersen}%
\address{University of Copenhagen\\
Dept. of Math.\\
Universitetsparken 5\\
2100 Copenhagen\\
Denmark.}

\thanks{The first two authors
were partially supported by NSF Grants DMS  11040934, DMS
0906233,  and NSF FRG grants DMS 
 0854774 and DMS 0853501}

\email{colding@math.mit.edu, minicozz@math.jhu.edu, and erik@math.ku.dk}

\begin{document}

\maketitle

\begin{abstract}
In this paper we will discuss how one may be able to use mean curvature flow to tackle some of the central problems in topology in $4$-dimensions.   

We will be concerned with smooth closed $4$-manifolds that can be smoothly embedded as a hypersurface in $\RR^5$.   
We begin with explaining why all closed smooth homotopy spheres can be smoothly embedded.  
After that we discuss what happens to such a hypersurface under the mean curvature flow.  If the hypersurface is in general or generic position before the flow starts, then we explain what singularities can occur under the flow and also why it can be assumed to be in generic position.  

The mean curvature flow is the negative gradient flow of volume, so any hypersurface flows through hypersurfaces in the direction of steepest descent for volume and eventually becomes extinct in finite time.  Before it becomes extinct, topological changes can occur as it goes through singularities.  Thus, in some sense, the topology is encoded in the singularities.   
\end{abstract}

\section{Introduction}

The aim of this paper is two-fold.   The bulk of it is spend on discussing mean curvature flow of hypersurfaces in Euclidean space and general manifolds and, in particular, surveying a number of very recent results about mean curvature starting at a generic initial hypersurface.  The second aim is to explain and discuss possible applications of these results to the topology of four manifolds.  In particular, the first few sections are spend to popularize a result, well known to older surgeons.  Namely, that any 
 closed smooth $4$-dimensional manifold homotopy equivalent to $\SS^4$ can be smoothly embedded as a hypersurface. We do this phrased in  modern language, but it is of course only a reformulation of a result due to Kervaire and Milnor \cite{KM}
in $\RR^5$.

With this theorem in hand, we spend the rest of the paper on mean curvature flow and, in the process, explain what kind of singularities can occur as a hypersurface evolves by the flow before it becomes extinct in finite time.

\vskip2mm
Surgery theory originated in the seminal paper of M. Kervaire and J. Milnor where they classified smooth manifolds homotopy equivalent to a sphere, and developed the basic surgery techniques in the simply connected case. These methods were taken up by W. Browder and S. Novikov. Browder's point of view to study the question of existence of a manifold homotopy equivalent to a given space whereas Novikov's approach was to investigate whether a given homotopy equivalence is homotopic to a diffeomorphism. D. Sullivan realized that existence and uniqueness are just two sides of the same question and formulated the theory in terms of a surgery exact sequence. Finally, the theory was vastly generalized to also treat non-simply connected manifolds by C.T.C. Wall.
The theory only works in dimensions bigger than four, but there are nevertheless a few things that work in all dimensions. In the first few sections we give a description of the surgery exact sequence and discuss low-dimensional phenomena. This leads to a description of the Kirby--Siebenmann obstruction to triangulate topological manifolds, it gives a fairly simply proof of topological invariance of Pontrjagin classes and finally discusses to what extent surgery works in dimension $4$.

We begin the second half of the paper with describing some of the fundamental classical results about mean curvature flow.  We explain what happens when closed curves evolve under the flow, that is the results of M. Gage, M. Grayson, and R. Hamilton.  We explain G. Huisken's results about evolution of convex hypersurfaces and why they become extinct in round points.  After that, we discuss the example of a dumbbell where the neck pinches off before the bells become extinct.  We turn next to Huisken's fundamental monotonicity formula and why it implies that the singularities are modeled by self-similar shrinkers that evolve under the mean curvature flow by rescaling.  We mention Angenent's, 
Kapouleas-Kleene-M\o ller's, and X. H. Nguyen's examples of an exotic self-similar shrinkers and the computer evidence of D. Chopp and T. Ilmanen for a whole host of other exotic examples.  After all of these fundamental and foundational results, we discuss some very recent results that are at the heart of this article.  Namely, that the only generic singularities, that is, the only singularities that cannot be perturbed away, are the simplest ones: The shrinking spheres and shrinking cylinders. We also discuss why it should be the case that if the initial hypersurface is in generic position, then under the flow one only sees generic singularities.   Finally, we mention briefly some very recent result about minimal entropy self-shrinkers and work in progress about a canonical neighborhood theorem describing the mean curvature flow in a space-time neighborhood of a generic singularity.

We are grateful to Dave Gabai for his interest and numerous stimulating discussions.

\input{cmp8}

\section{Mean curvature flow}

Suppose that $M$ is a closed hypersurface in $\RR^{n+1}$ and $M_t$ is a variation of $M$.  That is, $M_t$ is a one-parameter family of hypersurfaces with $M_0=M$.  An easy computation shows the following first variation formula for volume
\begin{equation}
 \frac{d}{dt} \,   \, \Vol \, (M_t)  = \int_{M_t}  \langle \partial_t \, x , H \, \nn \rangle  \, . 
 \notag
\end{equation}
Here $x$ is the position vector, $\nn$ the unit normal, and $H$ the mean curvature scalar given by 
\begin{equation}
H=\text{div}_{M} (\nn)  = \sum_{i=1}^n \langle \nabla_{e_i} \nn , e_i \rangle \, , \notag
\end{equation}
where $e_i$ is an orthonormal frame for $M$.  Equivalently, 
$H$ is the sum of the principal curvatures of $M$.
 With this normalization, $H$ is $n/R$ on the round $n$-sphere of radius $R$.  
 
 If we think of the volume as a function or functional on hypersurfaces, then it follows from the first variation formula that the gradient of volume is
 \begin{equation}
 \nabla\, \Vol=H\,\nn\, .\notag
 \end{equation}

The rest of this paper is about the negative gradient flow of volume (on the space of hypersurfaces); so the hypersurface is moving in the direction of steepest descent for volume.  The flow is called the mean curvature flow or MCF for short and a one-parameter family of hypersurfaces  $M_t \subset \RR^{n+1}$  flows by mean curvature precisely if 
\begin{equation}
	\partial_t \, x = - H \, \nn \, .\notag 
\end{equation}
In words, under the mean curvature flow, a hypersurface locally moves in the direction where the volume element decreases the fastest.    The flow has the effect of contracting a closed hypersurface, eventually leading to its extinction in finite time.

An easy consequence of the first variation is that if $M_t$ flows by mean curvature, then
\begin{equation}
	\frac{d}{dt} \,   \Vol (M_t) = -\langle \nabla \,\Vol, \nabla \,\Vol\rangle=- \int_{M_t} \, H^2  \, . \notag
\end{equation}

Our chief interest here is what happens before a hypersurface becomes extinct.  Is it possible to bring the hypersurface in general position so that one can describe and classify the changes that it goes through?  Is it possible to piece together information about the original hypersurface from the changes that it goes through under the flow?  In what follows, we will discuss some of the known results addressing these questions.

 \begin{figure}[htbp]
\centering\includegraphics[totalheight=.4\textheight, width=1\textwidth]{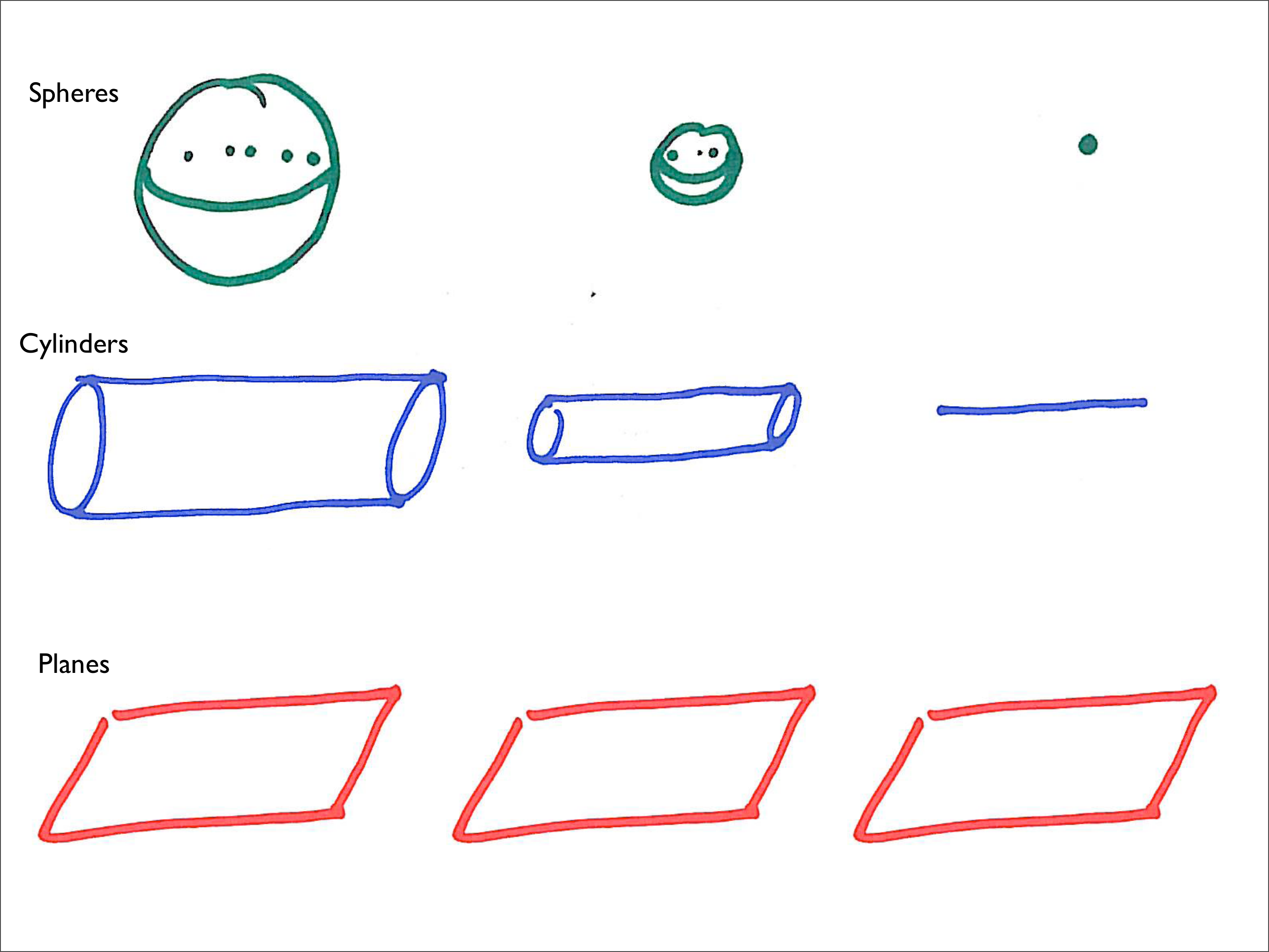}
\caption{Cylinders, spheres and planes are self-similar solutions of the mean curvature flow.
The shape is preserved, but the scale changes with time.}   
  \end{figure}

\subsection{Curve shortening flow}

When $n=1$ and the hypersurfaces are curves, the mean curvature flow is usually called the curve shortening flow.
Building on earlier work of Gage,
Gage and  Hamilton, \cite{GH}, showed in 1986 that  simple closed convex curves remain smooth and convex under the curve shortening flow and, eventually, become extinct in a ``round point''.
 More precisely, they showed that the flow became extinct in a point and if the flow is rescaled to keep the enclosed area constant, then the resulting curves converge to a round circle.  
A year later in 1987 M. Grayson, \cite{G}, showed that any simple closed curve eventually becomes convex
  under the curve shortening flow.
Thus, by the result of Gage-Hamilton, it becomes extinct in a ``round point''.

\begin{figure}[htbp]
\centering\includegraphics[totalheight=.35\textheight, width=.95\textwidth]{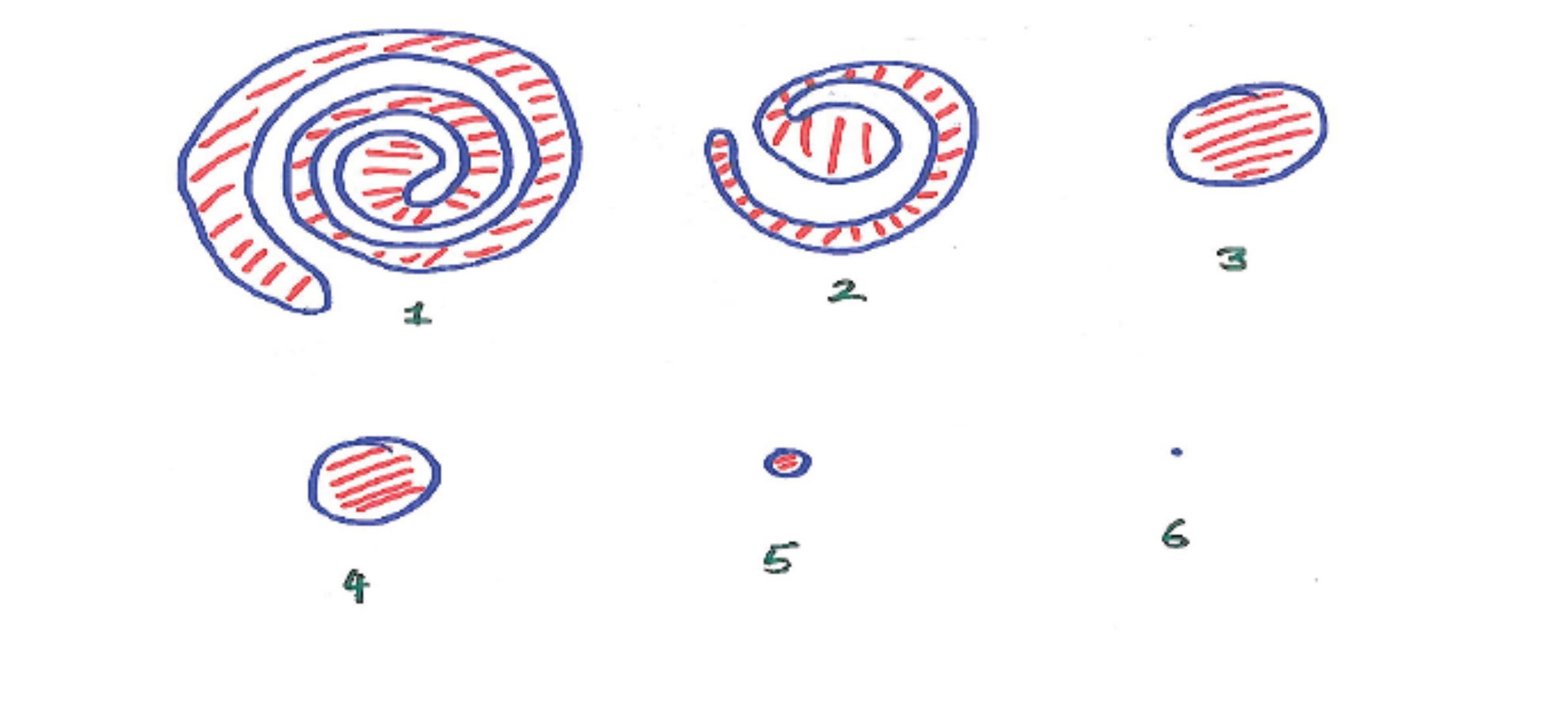}
\caption{The snake manages to unwind quickly enough to  become convex before extinction.}   
  \end{figure}

\subsection{The maximum principle}
An easy argument shows that the maximum principle holds for the mean curvature flow.  This has a number of consequences; among them the following key facts; see also Figure \ref{f:MaxPrinc}:
\begin{enumerate}
\item If two closed hypersurfaces are disjoint, then they remain disjoint under MCF.\label{e:item1}
\item If the initial hypersurface is embedded, then it remains embedded under MCF.  
\item If a closed hypersurface is convex, then it remains convex under MCF.
\item Likewise, mean convexity (i.e, $H\geq 0$) is preserved under MCF.
\end{enumerate}

\begin{figure}[htbp]
\centering\includegraphics[totalheight=.35\textheight, width=.95\textwidth]{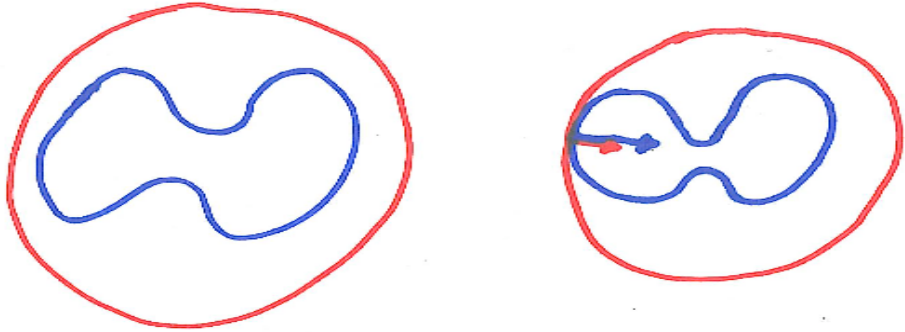}
\caption{By the maximum principle initially disjoint hypersurfaces remain disjoint under the flow.  To see this argue by contradiction and suppose not.  Look at the first time where they have contact.  At that time and point in space the inner evolves with greater speed, hence, right before they must have crossed, contradicting that it was the first time of contact.}   \label{f:MaxPrinc}
  \end{figure}

\vskip2mm
It follows from the avoidance property \eqr{e:item1} above that any closed hypersurface must  become extinct under the flow before the extinction of a large sphere containing the initial hypersurface.   For shrinking curves, Grayson proved that the singularities are trivial.  In higher dimensions, as we will see, the situation is much more complicated.  However, 
 one can define the flow through singularities. This was done by Sethian-Osher on the numerical side and Brakke, Evans-Spruck and Chen-Giga-Goto on the  theoretical side.   Obviously, as long as the flow stays smooth,  the evolving hypersurfaces are diffeomorphic.  Thus any topological change comes from singularities.   A major point of the rest of this paper is to discuss the singularities that typically occur.

In 1989, Grayson showed that his result for curves does not extend to surfaces.  In particular, he showed that 
a dumbbell with a sufficiently long and narrow bar will develop a pinching singularity before extinction.  A later proof was given by Angenent, \cite{A}, using the shrinking donut, that we will discuss shortly, and the avoidance property \eqr{e:item1}; see figure \ref{f:figAngenent} where Angenent's argument is explained.
Figures \ref{f:figdb1} to \ref{f:figdb2} show $8$ snapshots in time of the evolution of a dumbbell.{\footnote{The figures were created by computer simulation by U. Mayer.}}
 
\begin{figure}[htbp]
    \begin{minipage}[t]{0.5\textwidth}
    \centering\includegraphics[totalheight=.12\textheight, width=.95\textwidth]{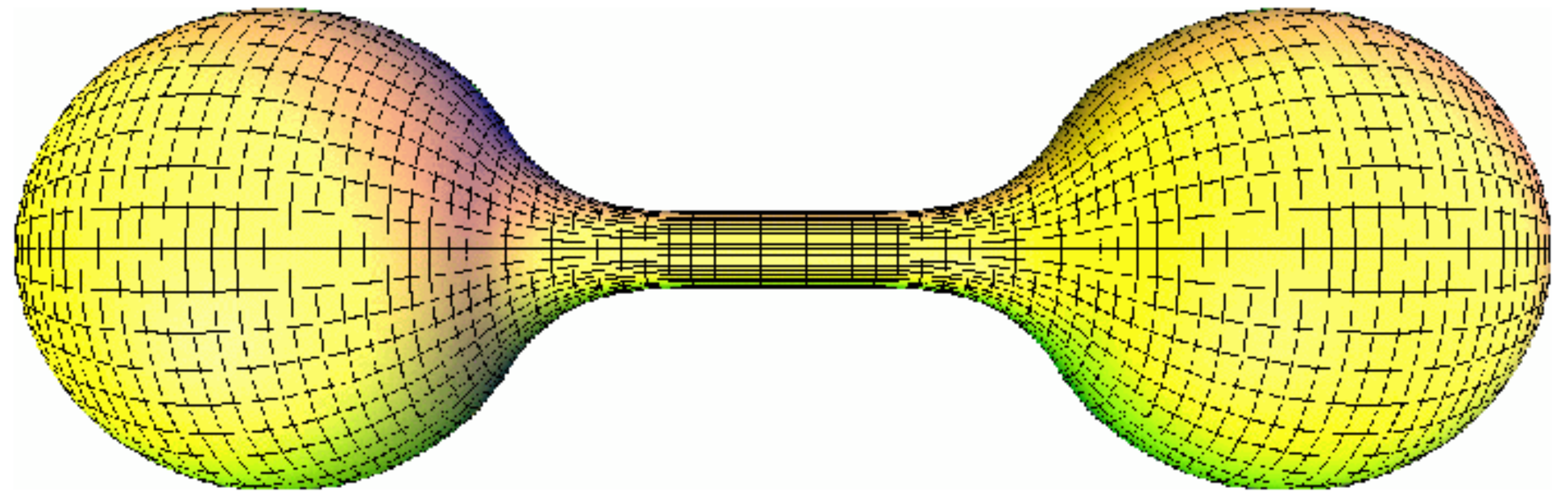}
    \end{minipage}\begin{minipage}[t]{0.5\textwidth}
    \centering\includegraphics[totalheight=.12\textheight, width=.95\textwidth]{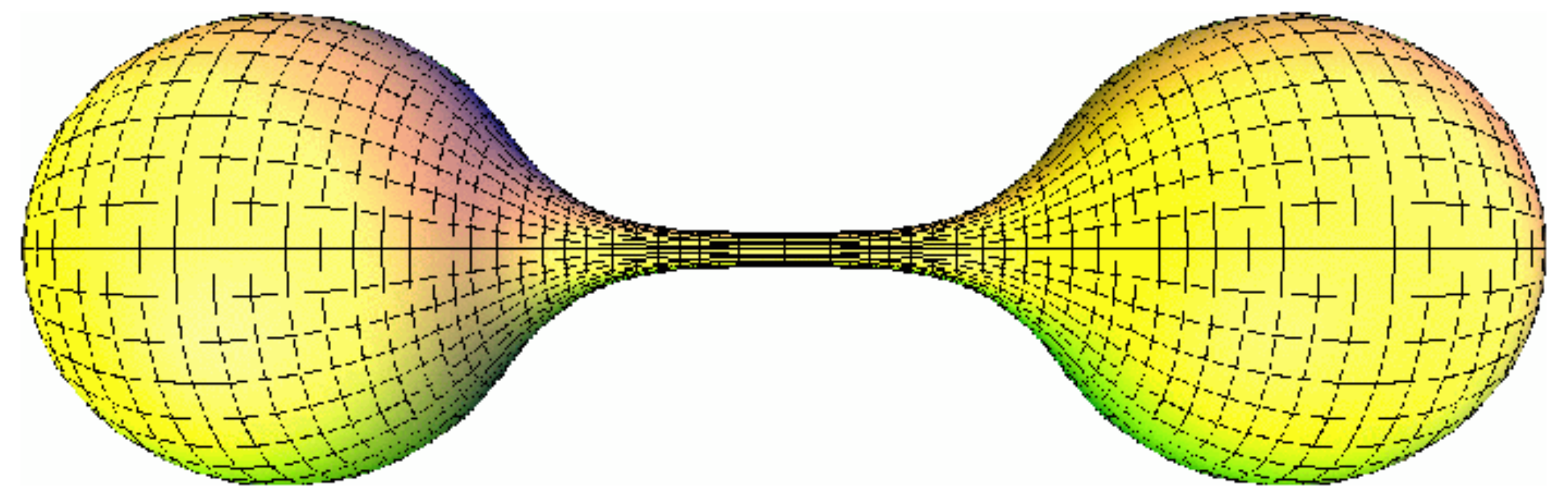}
    \end{minipage}
    \caption{Grayson's dumbbell; initial surface and step 1.}   \label{f:figdb1}
\end{figure}

\begin{figure}[htbp]
    \begin{minipage}[t]{0.5\textwidth}
    \centering\includegraphics[totalheight=.12\textheight, width=.95\textwidth]{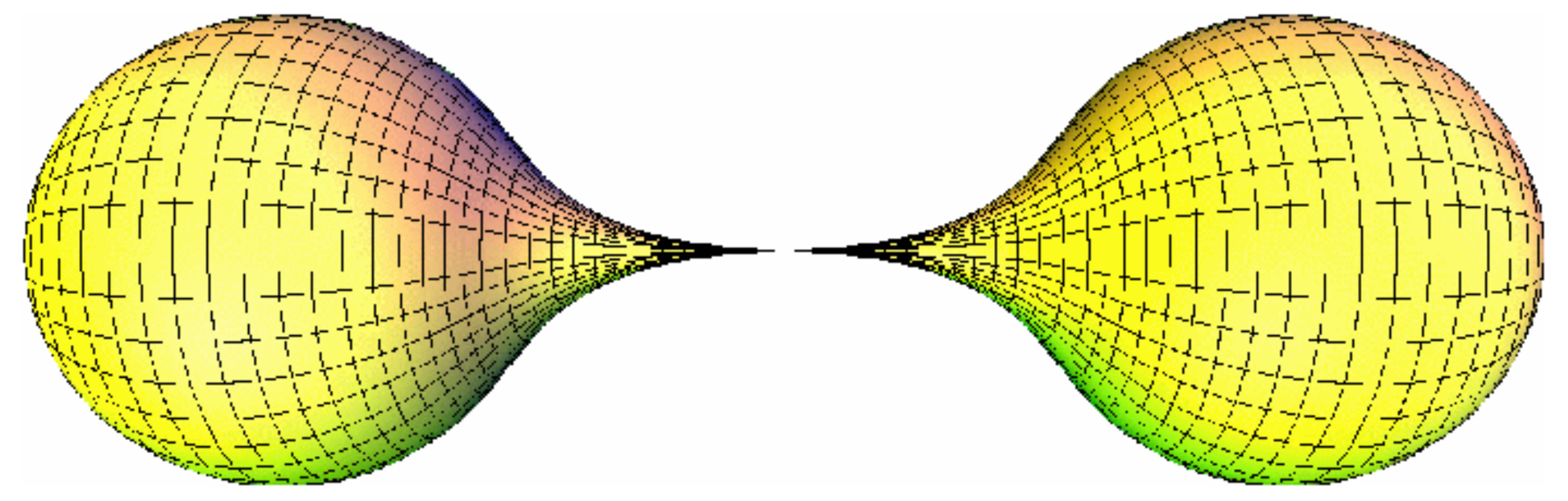}
    \end{minipage}\begin{minipage}[t]{0.5\textwidth}
    \centering\includegraphics[totalheight=.12\textheight, width=.95\textwidth]{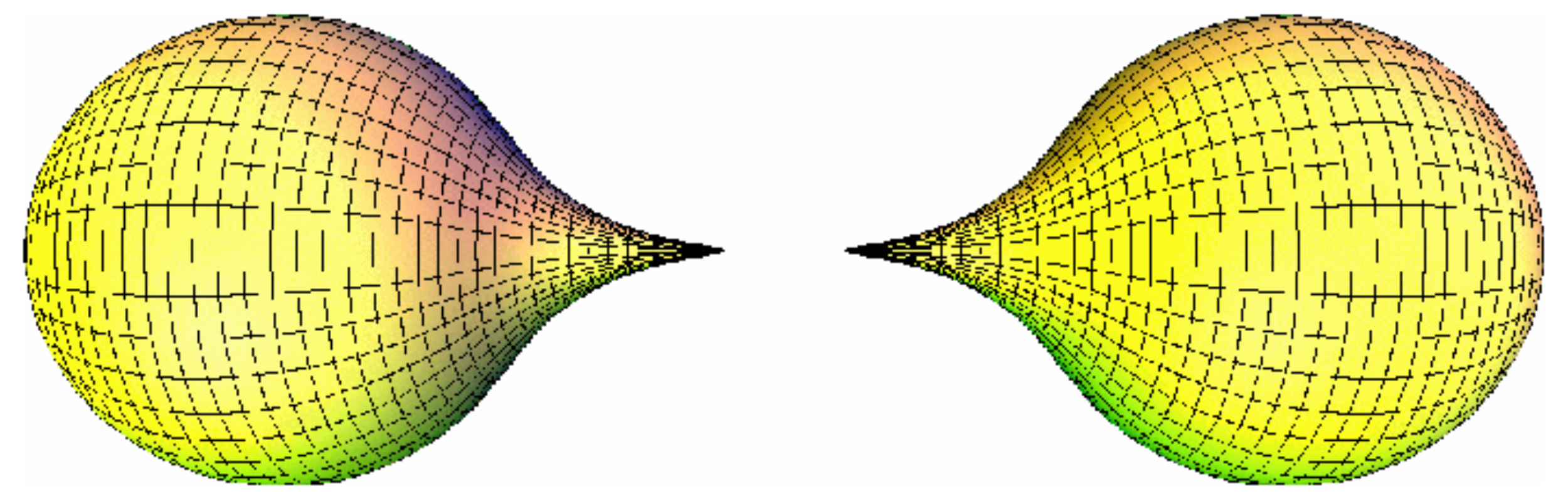}
    \end{minipage}
    \caption{The dumbbell; steps 2 and 3.}   
\end{figure}

\begin{figure}[htbp]
    \begin{minipage}[t]{0.5\textwidth}
    \centering\includegraphics[totalheight=.12\textheight, width=.95\textwidth]{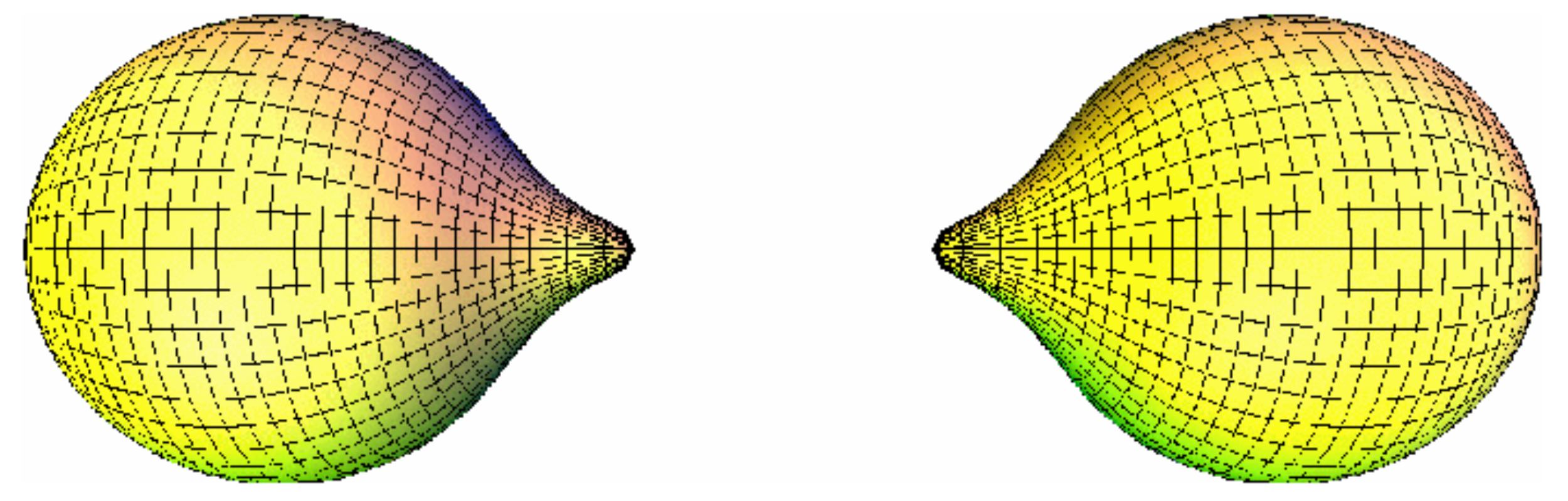}
    \end{minipage}\begin{minipage}[t]{0.5\textwidth}
    \centering\includegraphics[totalheight=.12\textheight, width=.95\textwidth]{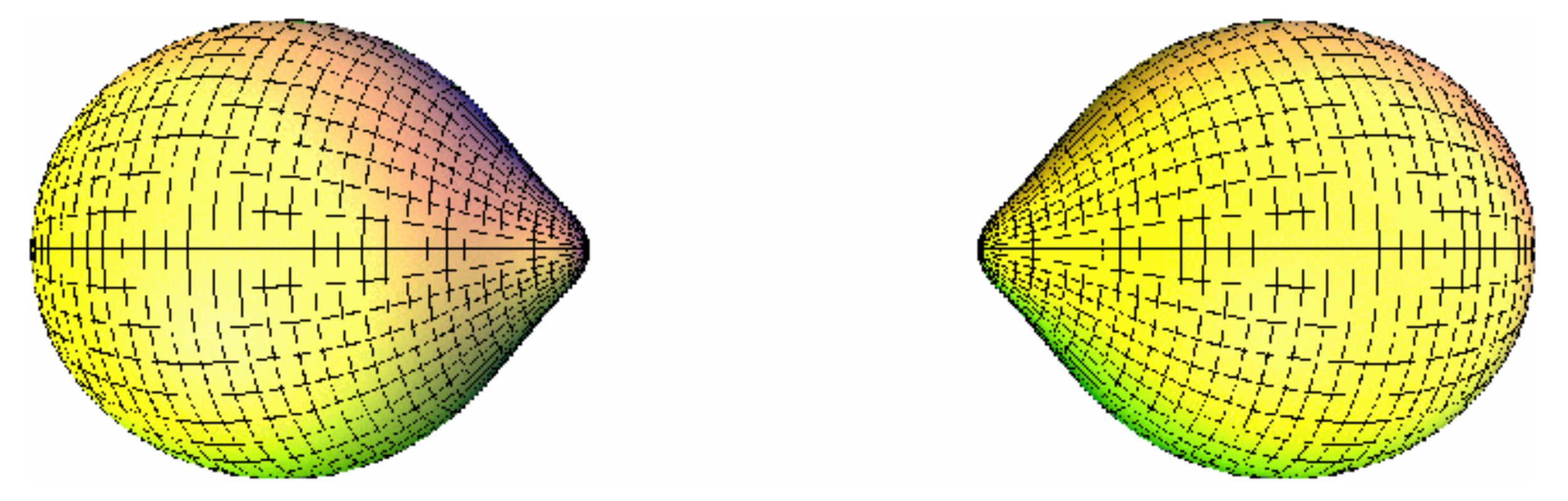}
    \end{minipage}
    \caption{The dumbbell; steps 4 and 5.}  
\end{figure}

\begin{figure}[htbp]
    \begin{minipage}[t]{0.5\textwidth}
    \centering\includegraphics[totalheight=.12\textheight, width=.95\textwidth]{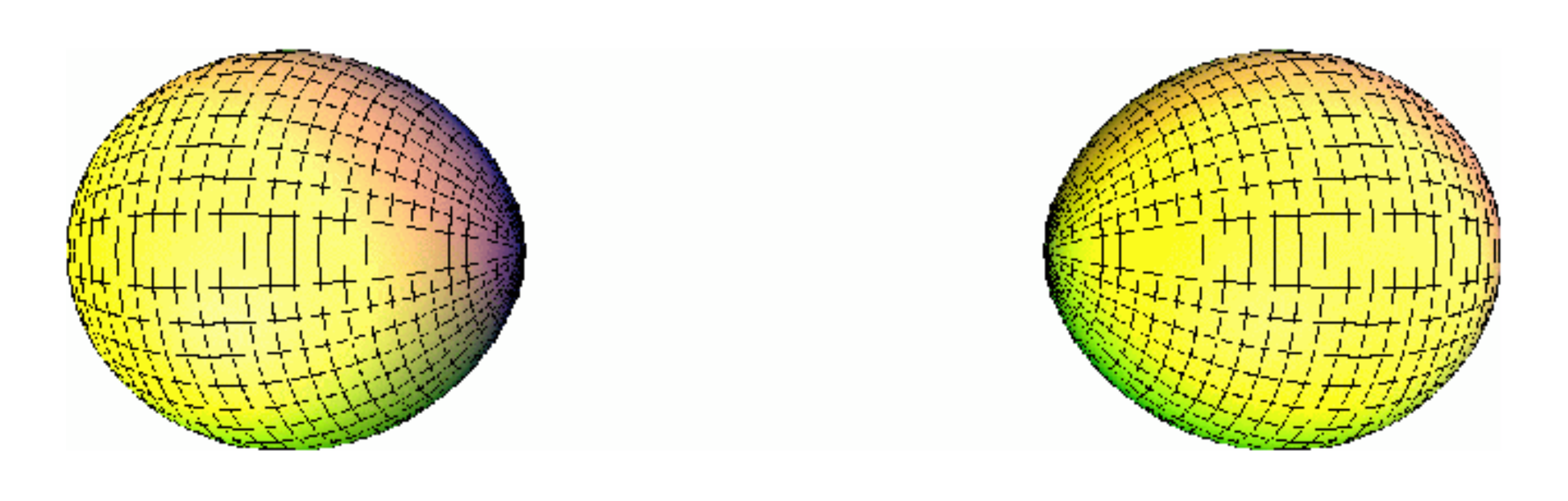}
    \end{minipage}\begin{minipage}[t]{0.5\textwidth}
    \centering\includegraphics[totalheight=.12\textheight, width=.95\textwidth]{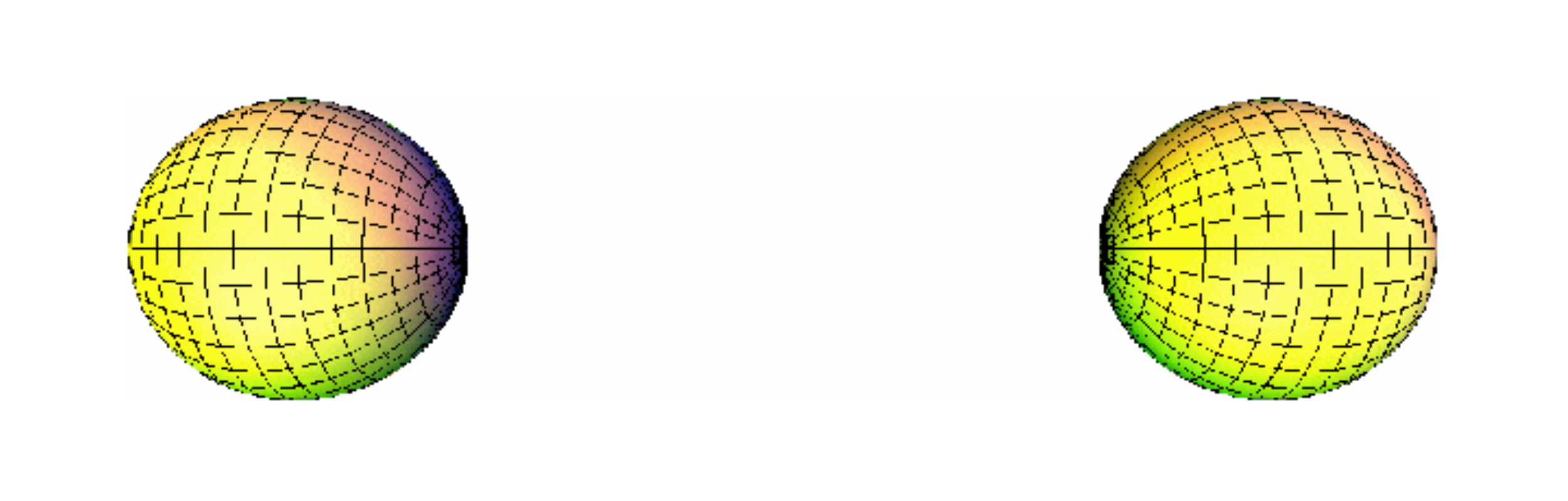}
    \end{minipage}
    \caption{The dumbbell; steps 6 and 7.}   \label{f:figdb2}
\end{figure}

\subsection{Self-similar shrinkers}

Recall that one of our interests here is in the topology of hypersurfaces.  The hope is that following the hypersurface as it evolves under the flow and understanding the changes that it goes through will give us information about the original hypersurface.  There are, however, hypersurfaces that only undergo trivial homothetic changes under the flow and, in particular, no improvement takes place.  Those are called self-similar shrinkers.  Examples are shrinking round spheres of radius $\sqrt{-2nt}$ where $t<0$ or the shrinking round cylinders $\SS^k\times \RR^{n-k}$ with radius $\sqrt{-2kt}$.  In both cases, the only change under the flow is a homothety. 

More precisely, let $M_t \subset \RR^{n+1}$ be a  one-parameter family of 
hypersurfaces flowing by MCF for $t< 0$.  We say that 
$M_t$ is a self-similar shrinker if $M_t = \sqrt{-t} \, M_{-1}$ for all $t< 0$.  Here, and in what follows, when $S$ is a subset of $\RR^{n+1}$ and $\lambda>0$ is a positive constant, then we let $\lambda\,S$ be the set $\{\lambda\,s\,|\,s\in S\}$ where the whole Euclidean space has been scaled by the factor $\lambda$.  

Note that   the self-similar shrinker becomes extinct at the origin in space at time $t=0$.  Since translations in space-time are isometries, we could equally well have considered surfaces that under the mean curvature flow evolve by homothety centered at a different point in space-time and become extinct at a different time and different point in space.  We will also refer to those as self-similar shrinkers and the equation for those is $M_t=\sqrt{t_0-t}\left(M_{t_0-1}-x_0\right)+x_0$.  However, when we consider a single self-similar solution, we will often assume for simplicity that it becomes extinct at the origin in time and space.  

In addition to shrinking round spheres and cylinders, then in 1992, 
 Angenent, \cite{A}, constructed by ODE methods a self-similar shrinking donut in $\RR^3$ together with similar higher dimensional examples.   Angenent's example was given by rotating a simple closed curve in the plane around an axis and thus had the topology of a torus.  In fact, numerical evidence suggests that, unlike for the case of curves,  a complete classification of self-shrinkers 
is impossible in higher dimensions as the examples appear to be so plentiful and varied; see for instance Chopp, \cite{Ch}, and Ilmanen, \cite{I2}, for numerical examples and the very recent rigorously constructed examples by gluing methods by  Kapouleas, Kleene, and M\o ller in \cite{KKM} and Nguyen in \cite{Nu}.

 \begin{figure}[htbp]
\centering\includegraphics[totalheight=.4\textheight, width=.75\textwidth]{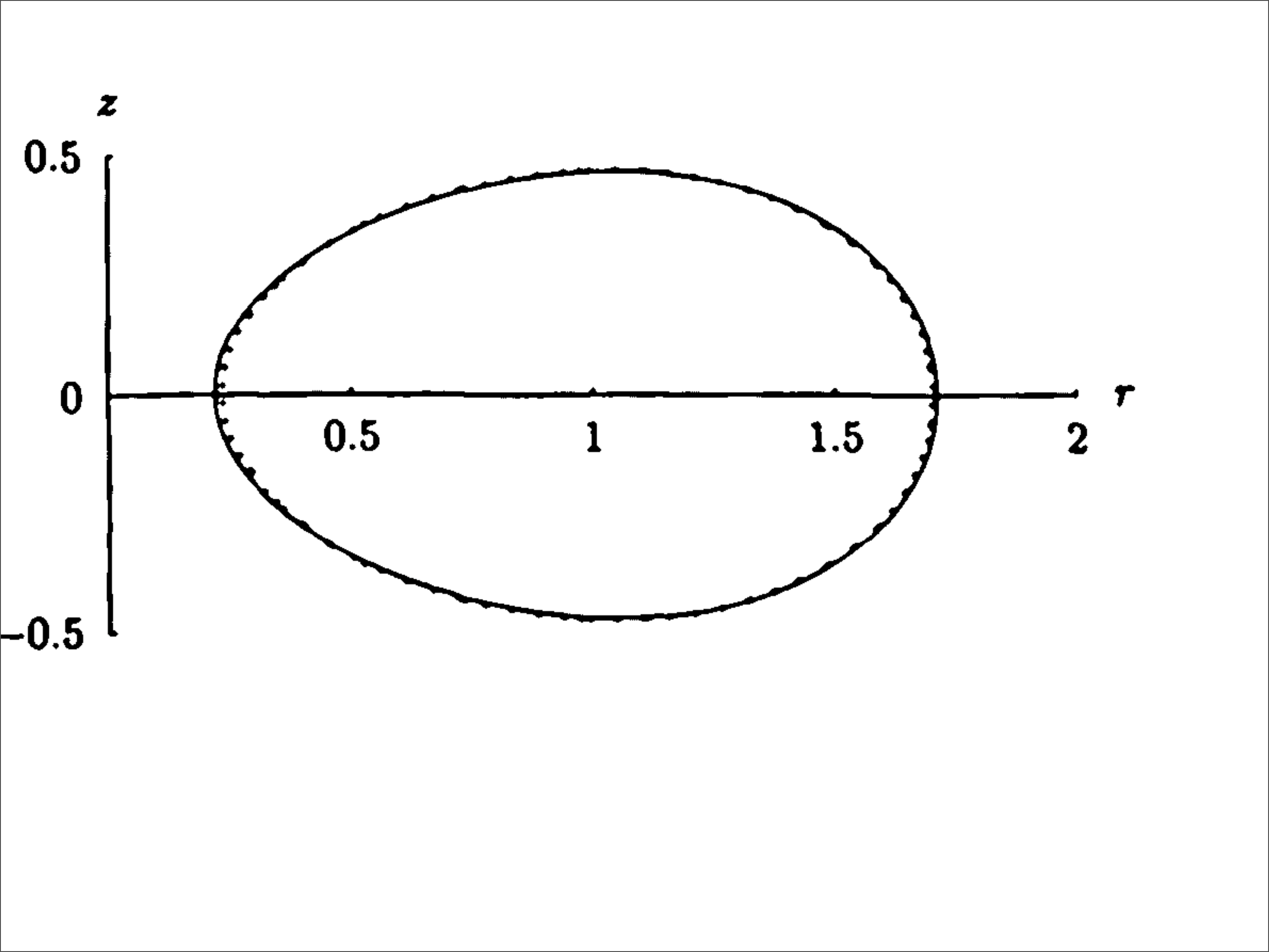}
\caption{Angenent's shrinking donut from numerical simulations of D. Chopp, \cite{Ch}.  The vertical $z$-axis is the axis of rotation
and the horizontal $r$-axis is a line of reflection symmetry.}   
  \end{figure}
  
  \begin{figure}[htbp]
 Ê Ê
\centering\includegraphics[totalheight=.4\textheight, width=.75\textwidth]{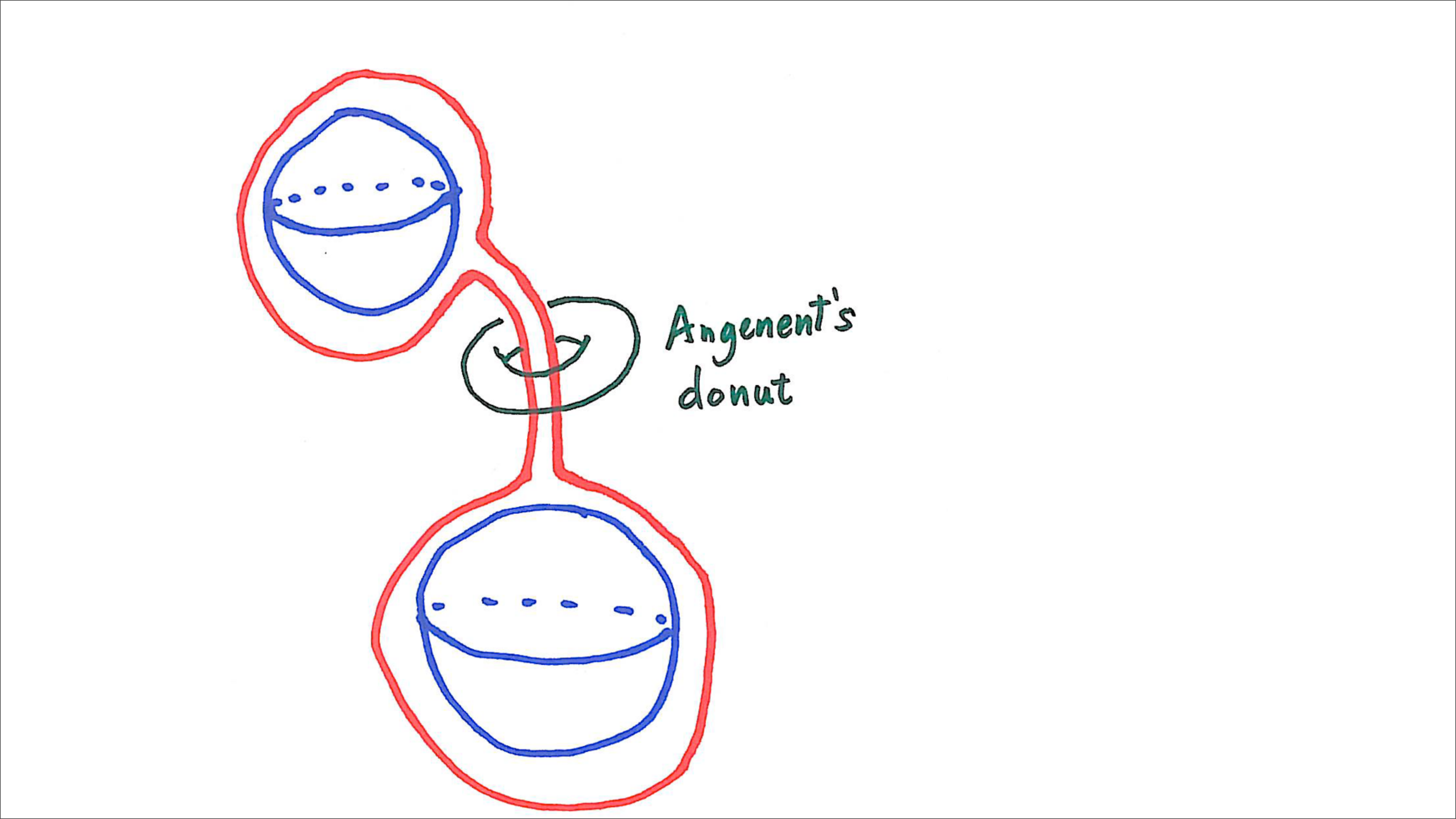}
\caption{Angenent's proof for why the neck of the dumbbell pinches off before the bells become extinct:  Enclose the neck with a small self-shrinking donut and place two round spheres inside the bells.  By the maxmum principle or avoidance property these four surfaces stay disjoint under the flow.  Since the donut becomes extinct before the two round spheres it follows that the neck pinches off before the bells become extinct.} Ê \label{f:figAngenent}
 Ê\end{figure}
 
  \begin{figure}[htbp]
\centering\includegraphics[totalheight=.4\textheight, width=.75\textwidth]{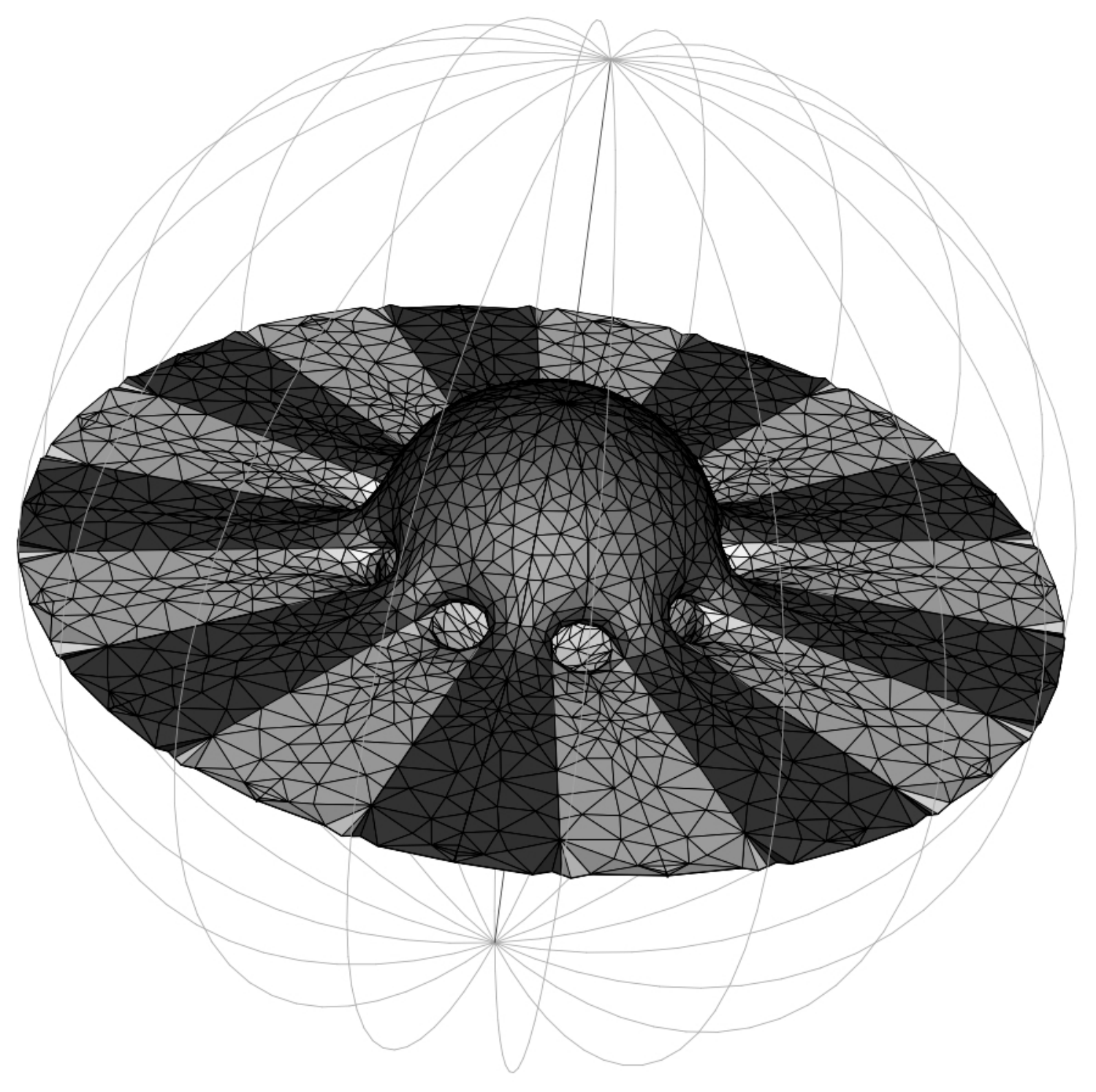}
\caption{The shrinker of genus 8
with 18 Scherk handles shown to exist by Kapouleas, Kleene, and M\o ller in \cite{KKM}, see also Nguyen, \cite{Nu}, for a similar shrinker.  Its existence had been conjectured by Ilmanen in \cite{I2}, where this picture is from.}   
  \end{figure}
 
\subsection{The self-shrinker equation}
An easy computation shows that a MCF $M_t$ is self-similar is equivalent to that
$M=M_{-1}$ satisfies the equation{\footnote{This equation differs by a factor of two from Huisken's definition of a self-shrinker; this is because Huisken works with the time $-1/2$ slice.}}
\begin{equation}  
H=\frac{\langle x,\nn \rangle}{2}\, .   \notag
\end{equation}
That is:  $M_t=\sqrt{-t}M_{-1}$  $\iff$  $M_{-1}$ satisfies $H=\frac{\langle x,\nn \rangle}{2}$. 

\begin{figure}[htbp]
    \begin{minipage}[t]{0.5\textwidth}
    \centering\includegraphics[totalheight=.3\textheight, width=1\textwidth]{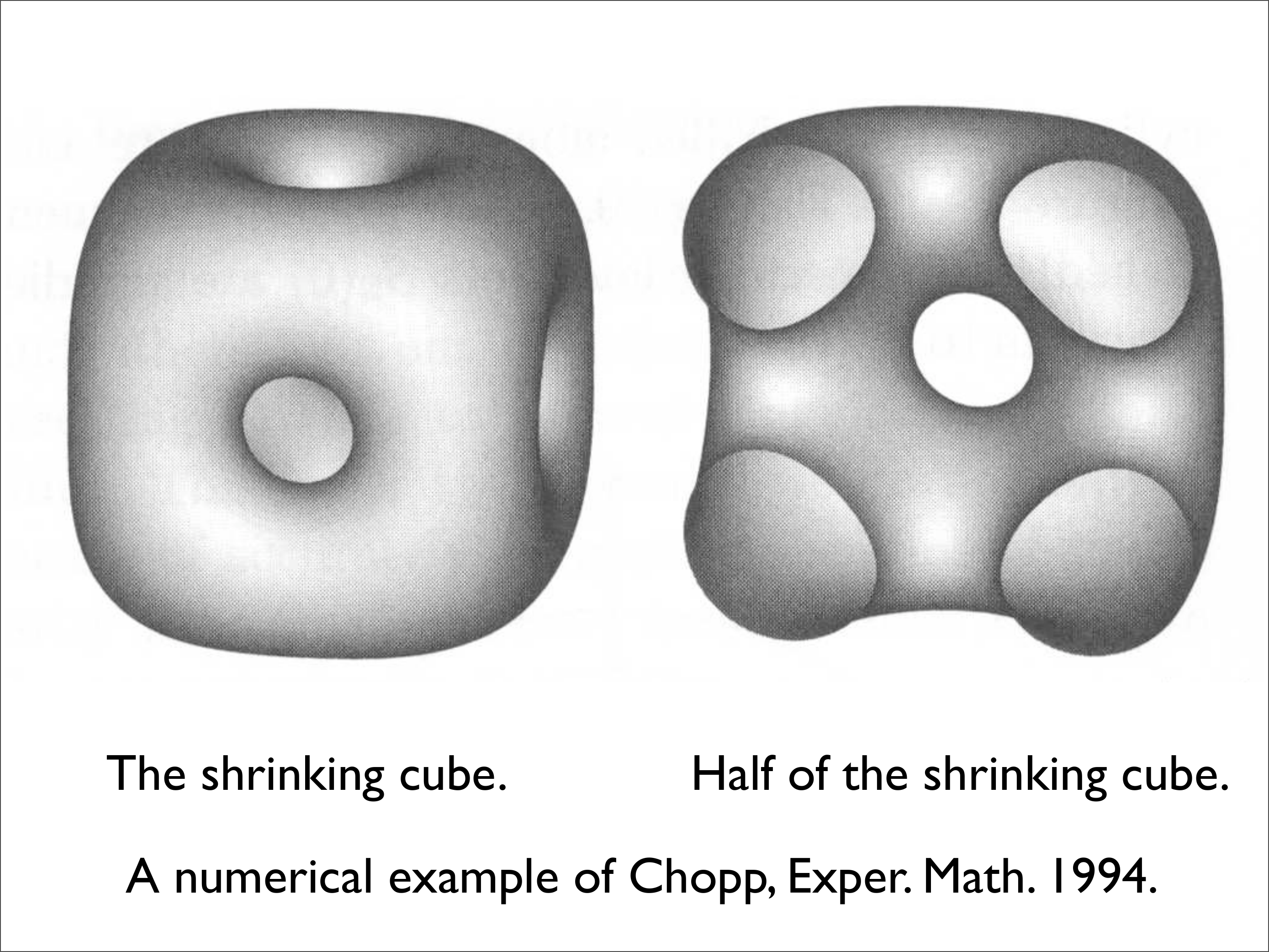}
    \caption{A closed numerical example of Chopp, \cite{Ch}.}  
    \end{minipage}\begin{minipage}[t]{0.5\textwidth}
    \centering\includegraphics[totalheight=.3\textheight, width=1\textwidth]{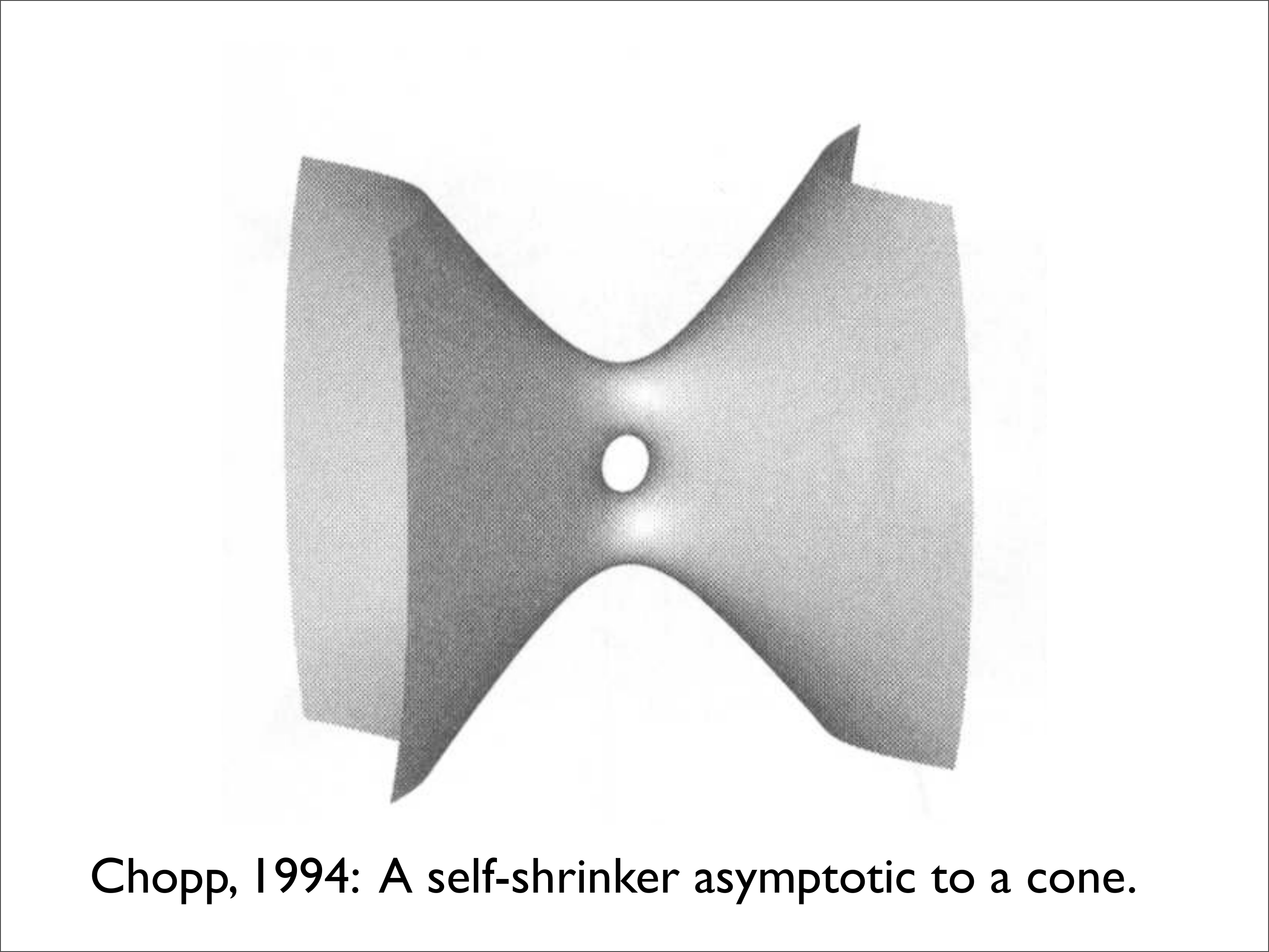}
    \caption{A non-compact numerical example of Chopp, \cite{Ch}.} 
    \end{minipage}
\end{figure}

The self-shrinker equation arises variationally in two closely related ways: as minimal surfaces for a conformally changed metric and as critical points for a weighted area functional.  We return to the second later, but state the first now:

\begin{Lem}
  $M$ is a self-shrinker
 $\iff$
 $M$ is a minimal surface in the metric 
 \begin{equation}
 	g_{ij}=\e^{-\frac{|x|^2}{2n}}\,\delta_{ij} \, .  \notag
\end{equation}
\end{Lem}
 
The proof follows immediately from the first variation.
Unfortunately, this metric on $\RR^{n+1}$ is not complete (the distance to infinity is finite) and
 the curvature blows up exponentially.

\subsection{Huisken's theorem about MCF of convex hypersurfaces}

In 1984, Huisken, \cite{H1}, showed that convexity is preserved under MCF and showed that closed convex hypersurfaces become round:

\begin{Thm}
(Huisken, \cite{H1})  Under MCF, every closed convex hypersurface in $\RR^{n+1}$  remains convex and eventually becomes extinct in a ``round point''.
\end{Thm}

 This is exactly analogous to the result of 
Gage-Hamilton for convex curves and was proven two years earlier.  However, Huisken's proof only works for $n>1$ as he shows that the hypersurfaces become closer and closer to being umbilic and that the limiting shapes are umbilic.  A hypersurface is umbilic if all of the eigenvalues of the second fundamental form are the same; this characterizes the sphere when there are at least two eigenvalues, but is meaningless for curves.

We will discuss generic singularities later, but want to point out here in the context of Huisken's and Gage-Hamilton's theorems that, as a consequence of the classification of the generic singularities together with a smooth compactness theorem for self-shrinkers (see \cite{CM1} and \cite{CM4}),
one gets that a 
 generic mean curvature flow in $\RR^3$  that disappears in a compact point does so in
a round point.

\vskip2mm
Using the maximum principle, one can show that various types of convexity are preserved under MCF.  Convexity means that every eigenvalue of $A$ has the same sign which, by necessity, has to be positive for a closed hypersurface.  There are weaker conditions that are also preserved under mean curvature flow and where significant results have been obtained.  Notably mean convexity, where a lot of interesting results have been obtained by Huisken-Sinestrari and White, \cite{HS1}, \cite{W}.  They have, in particular, studied the singularities that the flow develops.  A good way of thinking of mean convexity is that the flow moves monotone inward.  In another direction, Huisken-Sinestrari, \cite{HS2}, have developed mean curvature flow with surgery assuming that the initial hypersurface is $2$-mean convex.  $2$-mean convex means that the sum of any two principal curvatures is positive.

 \section{Singularities for MCF}

We will  now leave convex hypersurfaces and go to the general case.  As Grayson's dumbbell showed, there is no higher dimensional analog of his theorem for curves.
The key for analyzing singularities is a blow up (or rescaling) analysis, based on two ingredients:  Monotonicity and rescaling. Rescaling allows one to magnify around a singularity by blowing up the flow to obtain a new flow that models the given singularity.  The second ingredient is a monotonicity formula that guarantees that the blow up or rescaled flow becomes simpler.  In fact,  we will see next that the limit of the rescaled flows is self-similar. 
  
\subsection{Huisken's monotonicity}

Let  $\Phi$ be the non-negative function on $\RR^{n+1} \times (-\infty,0)$ defined by
 \begin{equation}
\Phi  (x,t)  = [-4\pi t]^{-\frac{n}{2}}
\,\e^{\frac{|x|^2}{4t}}\, ,
\end{equation}
and set
$\Phi_{(x_0,t_0)} (x,t)= \Phi (x-x_0,t-t_0)$.  So $\Phi_{(x_0,t_0)}$ is obtained from $\Phi$ by  translating in space-time.  In 1990, G. Huisken proved the following monotonicity formula for mean curvature flow, \cite{H2}, \cite{E}:

\begin{Thm}	\label{t:huiskenmon}
(Huisken, \cite{H2})
  If $M_t$ is a solution to
the MCF, then
\begin{equation}    \label{e:huisken2}
\frac{d}{dt} \int_{M_t} \Phi_{(x_0,t_0)}=
-\int_{M_t} \left| H\nn- \frac{(x-x_0)^{\perp}
}{2\,(t_0-t)}\right|^2 \,\Phi_{(x_0,t_0)}  \, .
\end{equation}
\end{Thm}

\vskip4mm
Huisken's density is the limit of $\int_{M_t} \Phi_{x_0,t_0}$ as $t\to t_0$.
That is,
\begin{equation}
    \Theta_{x_0, t_0} = \lim_{t\to t_0} \int_{M_t}\Phi_{x_0,t_0}\, ;
\end{equation}
this limit exists by the monotonicity \eqr{e:huisken2} and the density is non-negative as the integrand $\Phi_{x_0,t_0}$ is non-negative.

A fundamental aspect of this is that 
 Huisken's Gaussian volume $\int_{M_t} \Phi$ is constant in time if and only if   $M_t$ is a self-similar shrinker with
 $$
 M_t = \sqrt{-t} \, M_{-1}  \, .
 $$

\subsection{Tangent flows}

If $M_t$ is a MCF, then for all fixed constants $\lambda>0$ one can obtain a new MCF by scaling space by $\lambda$ and scaling time by $\lambda^{2}$.  The different scaling in time and space comes from that MCF is a parabolic equation where time accounts for one derivative and space for two just like the ordinary heat equation.  This type of scaling is usually referred to as parabolic scaling and it guarantees that the new one-parameter family also flows by MCF.  Precisely, parabolic scaling of $M_t$ with a constant $\lambda > 0$ is the new one-parameter family
\begin{equation}
	\tilde{M}_t = \lambda \, M_{\lambda^{-2} \, t}  \, .  \notag
\end{equation}
When $\lambda$ is large, this magnifies a small neighborhood of the origin in space-time.

If we now take a sequence $\lambda_i \to \infty$ and let $M^i_t = \lambda_i \, M_{\lambda_i^{-2} \, t}$, 
then
Huisken's monotonicity gives uniform Gaussian area bounds on the rescaled sequence. Combining this with Brakke's weak compactness theorem for mean curvature flow, it follows that 
a subsequence of the $M_t^i$ converges to a limiting flow $M^{\infty}_t$  (cf., for instance, \cite{W} and  \cite{I2}).  Moreover, Huisken's monotonicity implies that the Gaussian area (centered at the origin) is now constant in time, so we conclude that $M^{\infty}_t$ is  a self-similar shrinker.
This $M^{\infty}_t$ is called a {\emph{tangent flow}} at the origin.  The same construction can be done at any point of space-time.

\subsection{Gaussian integrals and the $F$-functionals}

We will next define a family of functionals on the space of hypersurfaces given by integrating Gaussian weights with varying centers and scales.
For $t_0>0$ and $x_0\in \RR^{n+1}$,   define   $F_{x_0,t_0}$  by
\begin{align}
F_{x_0 , t_0} (M)  = (4\pi t_0)^{-n/2} \, \int_{M} \, \e^{-\frac{|x- x_0|^2}{4t_0}} \, d\mu = \int_{M}\Phi_{x_0,t_0}(\cdot,0)\, .   \notag
\end{align}
We will think of $x_0$ as being the point in space that we focus on and $t_0$ as being the scale.  By convention, we set
$F=F_{0,1}$.

\subsection{Critical points for the $F$-functional}

 We will say that $M$ is a critical point for $F_{x_0 , t_0} $ if it is simultaneously critical with respect to variations in all three parameters, i.e.,   variations in $M$ and all variations in $x_0$  and   $t_0$.  Strictly speaking, it is the triplet $(M, x_0 , t_0)$ that is a critical point of $F$, but we will refer   to $M$ as a critical point of $F_{x_0 , t_0}$.
 The next proposition shows that $M$ is a critical point for $F_{x_0 , t_0} $ if and only if it is the time $-t_0$ slice of a self-shrinking solution of the mean curvature flow that becomes extinct at the point $x_0$ and time $0$.

 \begin{Pro}	\label{p:critall}
 (Colding-Minicozzi, \cite{CM1})
$M$ is a critical point for $F_{x_0 , t_0} $ if and only if  $M$ is a self-shrinker becoming extinct at the point $x_0$ in space and at time $t_0$ into the future.
\end{Pro}

\subsection{$F$-Stable or index $0$ critical points}  

A closed self-shrinker is said to be $F$-stable or just stable if, modulo translations and dilations, the second derivative of the $F$-functional is non-negative for all variations at the given self-shrinker, see \cite{CM1} for the precise definition as well as the definition of stability for non-compact self-shrinkers.    In \cite{CM1} it was shown that the round sphere is the only closed stable self-shrinker, i.e., closed index $0$ critical point for the $F$-functional modulo translations and dilations; see more in a later section.
 
 There are two equivalent ways of formulating the stability precisely for a closed self-shrinker.   We explain both since each way of thinking about stability has its advantages.  The first makes use of the whole family of $F$-functionals and is the following: 
 \begin{quote}
 A closed self-shrinker is said to be $F$-stable if for every one-parameter family of variations $\Sigma_s$ of $\Sigma$ (with $\Sigma_0 = \Sigma$), there exist variations $x_s$ of $x_0$ and $t_s$ of $t_0$ that make
 $F'' = \left( F_{x_s , t_s}(\Sigma_s) \right)'' \geq 0$ at $s=0$.
 \end{quote}
 The other (obviously equivalent) way of thinking about stability is where we think of a single $F$-functional and mod out by translations and dilations.  This second way will be particularly useful   in a later section when we discuss the dynamics of the flow near a closed unstable self-shrinker. 
 \begin{quote}
 A closed self-shrinker is said to be $F$-stable if for every one-parameter family of variations $\Sigma_s$ of $\Sigma$ (with $\Sigma_0 = \Sigma$), there exist variations $x_s$ of $0$ and $\lambda_s$ of $1$ that make
 $F'' = \left( F(\lambda_s\,\Sigma_s+x_s) \right)'' \geq 0$ at $s=0$.
 \end{quote}
 
 \begin{Thm}
 (Colding-Minicozzi, \cite{CM1}).  
 In $\RR^{n+1}$ the round sphere $\SS^n$ is the only closed smooth $F$-stable self-shrinker.
 \end{Thm}

\section{Generic singularities}
 
If $M_t$ flows by mean curvature and $t>s$, then Huisken's monotonicity formula gives 
\begin{equation}	\label{e:huiskenF}
		F_{x_0 , t_0} ( M_t)  \leq  F_{x_0 , t_0 + (t-s)} (M_s)  \, .
\end{equation}
Thus, we see that a fixed $F_{x_0,t_0}$ functional is not monotone under the flow, but the supremum over all of these functionals is monotone.    We call this invariant
 the entropy and denote it by
\begin{align}	\label{e:entropy}
\lambda (M) = \sup_{x_0 , t_0} \, F_{x_0 , t_0} (M)  \, .  
\end{align}
The entropy has three key properties:
\begin{enumerate}
\item $\lambda$ is invariant under dilations, rotations, and translations. \label{e:key1}
\item $\lambda (M_t)$ is non-increasing under MCF.\label{e:key2}
\item If $M$ is a self-shrinker, then $\lambda (M) = F_{0,1} (M) = \Theta_{0,0}$.
\end{enumerate}

\vskip2mm
A consequence of \eqr{e:key1} is, loosely speaking, that the entropy coming from a singularity  is independent of the time when it occurs, the point where it occurs, and even of the  scale at which the flow starts to resemble the singularity.

Note also that one way of thinking about \eqr{e:key2} is that $\nabla \,\Vol$ and $\nabla \,\lambda$ point toward the same direction 
in the sense that $\langle \,\nabla \Vol, \nabla \,\lambda\rangle\geq 0$.  We will use this later.


\subsection{How entropy is used}
The main point about $\lambda$ is that it can 
 be used to rule out certain singularities because
 of the monotonicity of entropy under MCF and its invariance under dilations:
 
 \begin{Cor}
If $M$ is a self-shrinker that occurs as a tangent flow for $M_t$ with $t>0$, then
\begin{equation}
  F_{0,1} (M) = \lambda (M) \leq \lambda (M_0) \, . \notag
 \end{equation}
 \end{Cor}
  
\subsection{Classification of entropy stable singularities}

The next theorem shows that the only singularities that cannot be perturbed away are the simplest ones.

\begin{Thm}	\label{c:nonlin1a}
(Colding-Minicozzi, \cite{CM1})
Suppose that $M^n \subset \RR^{n+1}$ is a smooth complete embedded
self-shrinker without boundary and with polynomial volume growth.
\begin{enumerate}
\item If $M$ is not equal to  $\SS^k\times \RR^{n-k}$, then
there is a graph $N$ over $M$ of a  function with arbitrarily small $C^m$ norm (for any fixed $m$) so that
  $\lambda ( N) < \lambda (M)$.
  \item If $M$ is not $\SS^n$ and does not split off a line, then   the function in (1) can be taken to have compact support.
  \end{enumerate}
  In  particular,   in either case, $M$  cannot arise as a tangent flow to the MCF starting from $N$.
\end{Thm}

Thus, spheres, planes and cylinders are the only generic self-shrinkers.
 
  \vskip2mm
  In fact, we have the following stronger result where the self-shrinker is allowed to have singularities:

 \begin{Thm}	\label{t:stronger}
 (Colding-Minicozzi, \cite{CM1})
 Theorem \ref{c:nonlin1a} holds when $n\leq 6$ and $M$ is smooth off of a singular set with locally finite $(n-2)$-dimensional Hausdorff measure.
 \end{Thm}
 
 \subsection{Self-shrinkers with low entropy/Gaussian surface area}
 
 Recall that the $F$-functional of a hypersurface $\Sigma$ of Euclidean space $\RR^{n+1}$ is the Gaussian surface area\footnote{Gaussian surface area has also been studied in convex geometry and in theoretical computer science.}
\begin{equation}
F(\Sigma)= \left( 4 \, \pi \right)^{ - \frac{n}{2} } \, \int_{\Sigma}\e^{-\frac{|x|^2}{4}}\, , 
\end{equation}
whereas the Gaussian entropy is the supremum over all Gaussian surface areas given by 
\begin{equation}
\lambda (\Sigma)=\sup  \, \,  \left( 4 \, \pi \, t_0 \right)^{ - \frac{n}{2} } \, \int_{\Sigma}\e^{-\frac{|x-x_0|^2}{4t_0}}\, .
\end{equation}
Here the supremum is taking over all $t_0>0$ and $x_0\in \RR^{n+1}$.  
Entropy is invariant under rigid motions and dilations.  Moreover, by section $7$ in \cite{CM1},  the entropy of a self-shrinker is equal to the value of $F$ and, thus,
 no supremum is needed.  Therefore by a result of Stone $\lambda (\SS^n)$ is decreasing in $n$ and
\begin{align}
	\lambda (\SS^1) = \sqrt{ \frac{2\pi}{\e} } \approx 1.5203 > \lambda (\SS^2) = \frac{4}{\e} \approx 1.4715 >
	\lambda (\SS^3) > \dots > 1 = \lambda (\RR^n) \, . 
\end{align}
Moreover, a simple computation shows that $\lambda (\Sigma \times \RR) = \lambda (\Sigma)$.

  \vskip2mm
 It follows from Brakke's regularity theorem for MCF that $\RR^n$ has the least entropy of any self-shrinker and, in fact, there is a gap to the next lowest.
A natural question is: Can one classify all low entropy self-shrinkers and if so what are those?  In \cite{CIMW} it is shown that the round sphere has the least entropy of any 
{\emph{closed}} self-shrinker.  

\begin{Thm}   \label{t:minentropy}
 (Colding-Ilmanen-Minicozzi-White, \cite{CIMW})
Given $n$, there exists $\epsilon =\epsilon (n) > 0$ so that if $\Sigma\subset \RR^{n+1}$ is a closed self-shrinker not equal to the round sphere, then $\lambda (\Sigma)\geq \lambda (\SS^n)+\epsilon$.  
Moreover, if $\lambda (\Sigma)\leq \min \{\lambda (\SS^{n-1}),\frac{3}{2}\}$, then 
$\Sigma$ is diffeomorphic to $\SS^n$.{\footnote{If $n > 2$, then $\lambda (\SS^{n-1}) < \frac{3}{2}$ and the minimum is unnecessary.}}
\end{Thm}

  Theorem \ref{t:minentropy} is suggested by the dynamical approach to MCF of
   \cite{CM1} and \cite{CM2} that we will discuss in more detail later on.   The  idea is that MCF starting at  a  closed $M$ becomes singular, the corresponding self-shrinker has lower entropy and,
 by  \cite{CM1},   the only   self-shrinkers that cannot be perturbed away are   $\SS^{n-k} \times \RR^k$ and   $
 \lambda (\SS^{n-k} \times \RR^k) \geq \lambda (\SS^n)$.    
  
The dynamical picture also suggests two closely related conjectures; the first is for any closed hypersurface and the second is for   self-shrinkers:

\begin{Con}	\label{c:minentropy}
(Colding-Ilmanen-Minicozzi-White, \cite{CIMW})
Theorem \ref{t:minentropy} holds with $\epsilon = 0$  for any closed hypersurface $M^n$ 
with $n\leq 6$.
\end{Con}

\begin{Con}	\label{c:c2}
(Colding-Ilmanen-Minicozzi-White, \cite{CIMW})
Theorem \ref{t:minentropy} holds   for any {\emph{non-flat}} self-shrinker $\Sigma^n \subset \RR^{n+1}$ 
with $n\leq 6$.

\end{Con}

Both conjectures are true for curves, i.e., when $n=1$.  The first conjecture
follows for curves by combining Grayson's theorem, \cite{G} (cf. \cite{GH}), and the monotonicity of $\lambda$ under curve shortening flow.  The second conjecture follows for curves from the classification of self-shrinkers by Abresch and Langer.

Conjecture \ref{c:c2} would allow us to carry out the outline above to show that any closed hypersurface has entropy at least that of the sphere, proving 
Conjecture \ref{c:minentropy}.

Furthermore, one could  ask which self-shrinker has the third least entropy, etc.  
It is easy to see that the entropy of the ``Simons cone'' over $\SS^k \times \SS^k$ in $\RR^{2k+2}$ is asymptotic to $\sqrt{2}$  as $k\to \infty$, which is also the limit of
$\lambda (\SS^{2k+1})$.  Thus, as the dimension increases, the Simons cones have lower entropy than some of the generalized cylinders.
For example,  the cone over $\SS^2 \times \SS^2$ has  entropy $\frac{3}{2}  < \lambda (\SS^1 \times \RR^4)$.  In other words, already for $n=5$, $\SS^k \times
\RR^{n-k}$ is not a complete list of the lowest entropy self-shrinkers.

  \section{Two conjectures about singularities of MCF}
  
  Thus far, we have mostly discussed smooth tangent flows (with the exception of Theorem \ref{t:stronger}).  However, tangent flows are not always smooth, but we have the
 following well known conjecture
(see page $8$ of \cite{I1}):

\begin{Con}  \label{c:codimconj}
Suppose that $M_0 \subset \RR^{n+1}$ is a smooth closed embedded hypersurface.  A time slice of any tangent flow of the MCF starting at $M_0$ has a singular set   of dimension at most $n-3$.
\end{Con}

 ÊObserve, in particular, that Theorem \ref{t:stronger}    classifies entropy stable self-shrinkers in $\RR^{n+1}$ assuming that they have  the smoothness of Conjecture \ref{c:codimconj}.  
 
 In \cite{I1}, Ilmanen proved that in $\RR^3$ tangent flows at the first singular time must be smooth,
although he left open the possibility of multiplicity. However, he conjectured that the multiplicity
must be one.  
 
 \subsection{Negative gradient flow near a critical point}

We are interested in the dynamical properties of mean curvature flow near a singularity.  Specifically we would like to show that the typical flow line or rather the mean curvature flow starting at the typical or generic hypersurface avoids unstable singularities.  Before getting to this, it is useful to recall the simple case of gradient flows near a critical point on a finite dimensional manifold.  Suppose therefore that $f: \RR^2 \to \RR$ is a smooth function with a non-degenerate critical point at $0$ (so $\nabla f (0) = 0$, but the Hessian of $f$ at $0$ has rank $2$).  The behavior of the negative gradient flow
$$
	(x' , y') = - \nabla f (x,y) 
$$
is determined by the Hessian of $f$ at $0$.  For instance, if $f(x,y)=\frac{a}{2}\,x^2+\frac{b}{2}\,y^2$ for constants $a$ and $b$, then the negative gradient flow solves the ODE's $x'=-a\,x$ and $y'=-b\,y$.  Hence, the flow lines are given by $x=\e^{-at}\,x(0)$ and $y=\e^{-bt}\,y(0)$.

The behavior  of the flow near a critical point depends on the index of the critical point, as is illustrated by the following examples:
\begin{enumerate}
\item[(Index 0):]
The function $f(x,y) = x^2 + y^2$ has a minimum at $0$.  The vector field is
$(-2x, -2y)$ and the flow lines are rays into the origin.  Thus every flow line limits to $0$.
\item[(Index 1):] The function $f(x,y) = x^2 - y^2$ has an index one critical point at $0$.  The vector field is
  $(-2x, 2y)$ and the flow lines are 
level sets of the function $h(x,y) = xy$.  
Only points where $y=0$ are on flow lines that limit to the origin. 
\item[(Index 2):]
 The function $f(x,y) = - x^2 - y^2$ has a maximum at $0$.  The vector field is
$(2x, 2y)$ and the flow lines are rays out of the origin.  Thus every flow line limits to $\infty$ and it is impossible to reach $0$.
\end{enumerate}

\vskip8mm
\includegraphics[totalheight=.35\textheight, width=.9\textwidth]{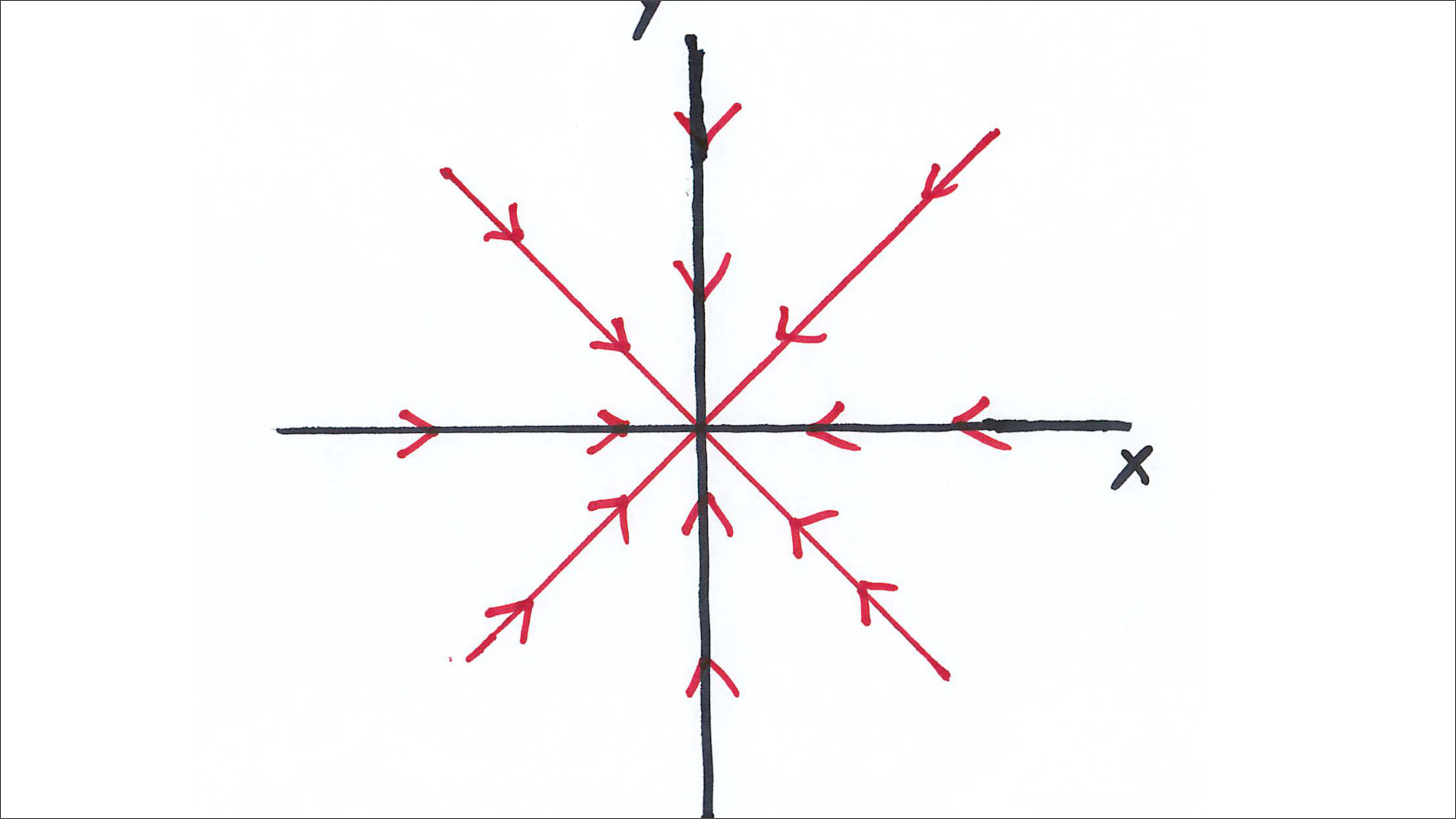}
\vskip1mm
\noindent
$f(x,y) = x^2 + y^2$ has a minimum at $0$.
Flow lines: Rays through the origin.

\vskip8mm
\includegraphics[totalheight=.35\textheight, width=.9\textwidth]{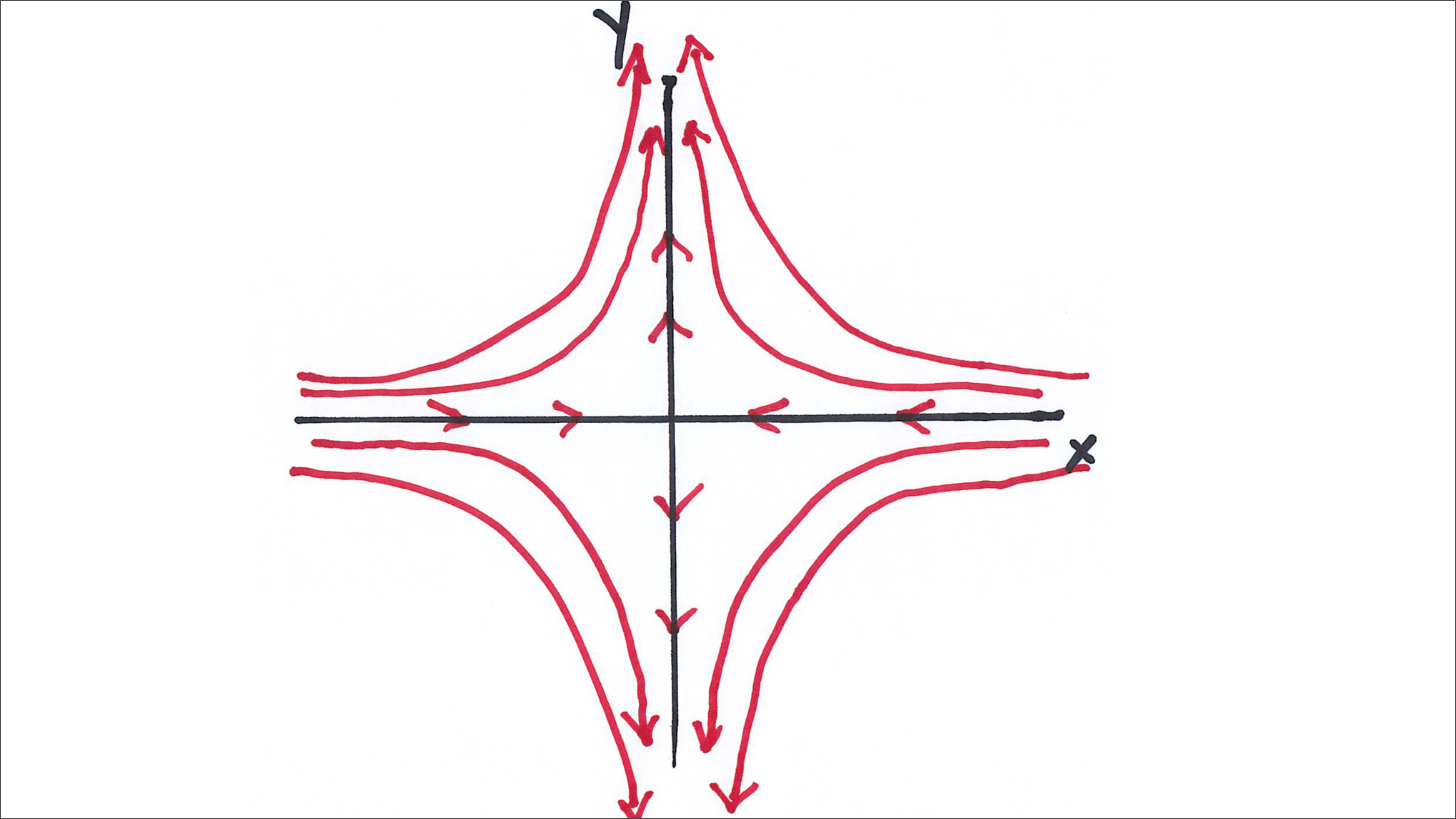}
\vskip1mm
\noindent
$f(x,y) = x^2 - y^2$ has an index one critical point at $0$.
Flow lines:  Level sets of   $ xy$.

\noindent
Only points where $y=0$ limit to the origin. 

\vskip2mm
Thus, we see that the critical point $0$ is ``generic'', or dynamically stable, if and only if it has index $0$.  When the index is positive, the critical point is not generic and a ``random'' flow line will miss the critical point.

The stable manifold for a flow near the fixed point is the set of points $x$ so that the flow starting from $x$ is defined for all time, remains near the fixed point, and converges to the fixed point as $t\to \infty$.  For instance, in the three examples above, the stable manifold is all of $\RR^2$, the $x$-axis, and the origin, respectively.  

\vskip4mm
It is also useful to recall what it means for directions to be expanding or contracting under a flow.  We will explain this in the example of the negative gradient flow of the function $f(x,y)=\frac{a}{2}\,x^2+\frac{b}{2}\,y^2$.  We saw already that the time $t$ flow is $\Psi_t (x,y)=(\e^{-at}\,x,\e^{-bt}\,y)$.  In particular, the time $1$ flow $\Psi=\Psi_1$ is the diagonal matrix with entries $\e^{-a}$ and $\e^{-b}$.  It follows, that if $a$ is negative, then the $x$ direction is expanding for the flow and if $a$ is positive, then the $x$ direction is contracting.  Likewise for $b$ and the $y$ direction.  

\vskip2mm
Consider next a slightly more general situation, where $f$ and $g$ are Morse functions.  We will assume that the gradients of $f$ and $g$   point towards the same direction meaning that 
\begin{equation}
\langle \nabla f,\nabla g\rangle \geq 0\, .\notag
\end{equation}
We will flow in direction of $-\nabla f$ and would like to claim that the typical flow line avoids  the unstable critical points of $g$.    

Later the volume function $\Vol$ will play the role of $f$ and the entropy $\lambda$ will play the role of $g$.   The assumption that $\nabla\, \Vol$ and $\nabla\, \lambda$ point toward the same direction is a consequence of Huisken's monotonicity formula, see the second key property of the entropy above.

To see this claim, consider an unstable critical point for $g$.  Unless it is also a critical point for $f$, then there is nothing to show as away from critical points of $f$ though each point there is only one flow line.  We can therefore assume that the unstable critical point of $g$ is also a critical point for $f$.  It now follows easily from the assumption that $f$ is monotone non-increasing along the negative gradient flow of $g$  that the given point is also an unstable critical point for $f$ and, hence, the claim follows.

\subsection{Flows beginning at a typical hypersurface avoid unstable singularities}

We saw above that, at least in finite dimensions,  a typical flow line of a negative gradient flow  of a function $f$ avoids unstable critical points of a function $g$ when their gradients toward the same direction.  Applying this to $f=\Vol$ and $g=\lambda$ together with the classification of entropy stable singularities (Theorems \ref{c:nonlin1a}
and \ref{t:stronger}) leads naturally to the following conjecture; we will in the next section discuss some of the very recent work from \cite{CM2} related to this conjecture:\footnote{Note that the only smooth hypersurface that is both a critical point for $\lambda$, i.e. is a self-shrinker, and is a critical point for $\Vol$, i.e. is a minimal surface, is a hyperplane.  Namely, any self-shrinker with $H=0$ is a cone as $x^{\perp}=0$, and hence, if smooth, is a hyperplane.  We shall not elaborate further on this, but it illustrates that the claim about typical flow lines is more subtle than the above heuristic argument indicates.  Namely, else one would have that the mean curvature flow starting at the typical or generic hypersurface does not become singular as points in space-time with tangent flow that is a hyperplane is a smooth point in space-time by a result of Brakke.  This is however clearly non-sensible as the flow becomes extinct in finite time and thus must develop a singularity.  In addition to this, then it follows, for instance, from Huisken's result, that any closed convex hypersurface becomes extinct in a round point, that shrinking spheres indeed are generic singularities.}

 \begin{Con}
 Suppose that $M^n \subset \RR^{n+1}$ is an embedded smooth closed hypersurface where $n\leq 6$, then there is a graph $N$ over $M$ of a  function with arbitrarily small $C^m$ norm (for any fixed $m$) so that for the MCF starting at $N$ all tangent flows at singularities are either shrinking round spheres $\SS^n$ or shrinking generalized cylinders $\SS^k \times \RR^{n-k}$.  
  \end{Con}
  
  \section{Dynamics of closed singularities}
  
  In this section we will discuss some very recent results related to the previous conjecture.  This work shows that, for generic initial data, the mean curvature flow never ends up in an unstable closed singularity.  The first step of this is to show that, near a closed unstable self-shrinker, the dynamics of the negative gradient flow of the $F$-functional looks exactly like an infinite dimensional version of the dynamics of the negative gradient flow of the function $f(x,y)=x^2-y^2$ near the unstable critical point $(x,y)=(0,0)$.
  
  We have already seen that the mean curvature flow is the negative gradient flow of volume.   
We have seen that singularities are modeled by their blow-ups, which are self-similar shrinkers and we have explained that the only smooth stable self-shrinkers are spheres, planes, and generalized cylinders (i.e., $\SS^k \times \RR^{n-k}$).    In particular, the round sphere is the only closed stable singularity for the mean curvature flow. 

Suppose that $M_t$ is a one-parameter family of closed hypersurfaces flowing by MCF. We want to analyze the flow near a singularity in space-time.  After translating, we may assume that the singularity occurs at the origin in space-time.  If we reparametrize and rescale the flow as follows $t\to M_{-\e^{-t}}/\sqrt{\e^{-t}}$, then we get a solution to the rescaled MCF equation.    The rescaled MCF is the negative gradient flow for the $F$-functional
\begin{align}
	F(\Sigma) = 	\left( 4 \, \pi \right)^{ - \frac{n}{2} } \, \int_{\Sigma} \e^{ - \frac{|x|^2}{4} } \, ,
\end{align}
where the gradient is with respect to the weighted inner product on the space of normal variations.    The fixed points of the rescaled MCF, or equivalently the critical points of the $F$-functional, are the self-shrinkers.   The rescaling   to get to the rescaled MCF   turns the question of the dynamics of the MCF near a singularity into a question of the dynamics near a fixed point for the rescaled flow.  
We can therefore treat the rescaled MCF as a special kind of dynamical system that is the gradient flow of a globally defined function and where the fixed points are the singularity models for the original flow.   
  
  The paper \cite{CM2} analyzes the behavior of the rescaled flow in a neighborhood of a closed unstable  self-shrinker.  Using this analysis, it is shown that generically one never ends up in an unstable closed singularity.     
  A key step  is to show that, in a suitable sense, ``nearly every'' hypersurface in a neighborhood of the unstable self-shrinkers is wandering or, equivalently, non-recurrent.  In contrast, in a small neighborhood of the round sphere, all closed hypersurfaces are convex and thus all become extinct in a round sphere under the MCF by a result of Huisken, \cite{H1}.  The point in space-time where a closed hypersurface nearby the round sphere becomes extinct may be different from that of the given round sphere.  This corresponds to that, under the rescaled MCF, it may leave a neighborhood of the round sphere but does so near a translation of the sphere.   Similarly, in a neighborhood of an unstable self-shrinker, there are closed hypersurfaces that under the rescaled MCF leave the neighborhood of the self-shrinker but do so in a trivial way, namely, near a translate of the given unstable self-shrinker.  This, of course, leads to no real change.  However,  it was shown in \cite{CM2}  that a typical closed hypersurface near an unstable self-shrinker not only leaves a neighborhood of the self-shrinker, but, when it does, is not close to a rigid motion or dilation of the given self-shrinker.   Thus, we have a genuine improvement, or at least a change, of the singularity.    Using that the rescaled MCF is a gradient flow, it was also shown in \cite{CM2} that once the flow leaves a neighborhood of a self-shrinker it will never return but wander off.  Together,  this not only gives a change, but an actual improvement.  
 
   \subsection{Dynamics near a closed self-shrinker}
   
   In this subsection, we will explain in what sense the dynamics of the negative gradient flow of the $F$-functional near a closed unstable self-shrinker looks like an infinite dimensional version of the dynamics of the negative gradient flow of the function $f(x,y)=x^2-y^2$ near the index $1$ critical point $(x,y)=(0,0)$.

Let $E$ be the Banach space of $C^{2,\alpha}$ functions on a smooth closed embedded hypersurface $\Sigma \subset \RR^{n+1}$ with unit normal $\nn$.  We are identifying $E$ with the space of $C^{2,\alpha}$ hypersurfaces near $\Sigma$ by mapping a function $u$ to its graph 
\begin{equation}
	\Sigma_u = \{ p + u(p) \, \nn (p) \, | \, p \in \Sigma \} \, .
\end{equation}
If $E_1 , E_2$ are subspaces of $E$ with $E_1 \cap E_2 = \{ 0 \}$ and that together span $E$, i.e., so that 
\begin{equation}
	E = \{ x_1 + x_2 \, | \, x_1 \in E_1 , x_2 \in E_2 \} \, ,
\end{equation}
then we will say that $E= E_1 \oplus E_2$ is a splitting of $E$.

\vskip2mm
The essence of the next theorem is that ``nearly every'' hypersurface in a neighborhood of the given unstable singularity leaves the neighborhood under the recaled MCF and, when it does, is not near a translate, rotation or dilation of the given singularity.  

\begin{Thm} 	\label{t:one}
(Colding-Minicozzi, \cite{CM2}).   
Suppose that $\Sigma^n \subset \RR^{n+1}$ is a smooth closed embedded self-shrinker, but is not a sphere.  There exists an open  neighborhood $\cO = \cO_{\Sigma}$ of $\Sigma$ and  a  subset $ W$ of $\cO$ so that:
\begin{itemize}
\item There is a splitting $E= E_1 \oplus E_2$  with  $\dim (E_1) > 0$ 
so that  $W$ is contained in the graph $(x,u(x))$ of a continuous mapping     $u: E_2 \to E_1$.
\item  If $\Gamma \in \cO \setminus W$, then the rescaled mean curvature flow starting at $\Gamma$ leaves $\cO$ and the orbit of $\cO$ under the group of conformal linear transformations of $\RR^{n+1}.  $\footnote{Recall that the group of conformal linear transformations of $\RR^{n+1}$ is generated by the rigid motions and the dilations.}
\end{itemize}
\end{Thm}

The space $E_2$ is, loosely speaking, the span of all the contracting directions for the flow together with all the directions tangent to the action of the conformal linear group.    It turns out that all the directions tangent to the group action are expanding directions for the flow.

Recall that the (local) stable manifold is the set of points $x$ near the fixed point so that the flow starting from $x$ is defined for all time, remains near the fixed point, and converges to the fixed point as $t\to \infty$.
Obviously, Theorem \ref{t:one} implies that the local stable manifold  is contained in $W$.  

There are several earlier results that analyze rescaled MCF  near a closed  self-shrinker, but all of these are for round circles and spheres which are stable under the flow.  The earliest are the global results of Gage-Hamilton, \cite{GH}, and Huisken, \cite{H1}, mentioned in an earlier section of this paper, showing that closed embedded convex hypersurfaces flow to spheres.   
 There is also a stable manifold theorem of Epstein-Weinstein, \cite{EW}, from the late 1980s for the curve shortening flow that also applies to closed immersed self-shrinking curves, but does not incorporate the group action.     In particular, for something to be in Epstein-Weinstein's stable manifold, then under the rescaled flow it has to limit into the given self-shrinking curve.  In other words, for a curve to be in their stable manifold it is not enough that it limit into a rotation, translation or dilation of the self-shrinking curve.  
 
 \subsection{The heuristics of the local dynamics}
 We will very briefly explain the underlying reason for this theorem about the local dynamics near a closed self-shrinker and why it is an infinite dimensional and nonlinear version of the simple finite dimensional examples we discussed earlier.  
 
 Suppose $\Sigma$ is a manifold and $h$ is a function on $\Sigma$.  
Let $w_i$ be an orthonormal basis of the Hilbert space $L^2(\Sigma,\e^h\,d\Vol)$, where the inner product is given by $\langle v,w\rangle=\int_{\Sigma} v\,w\,\e^{h}\,d\Vol$.  For constants $\mu_i\in \RR$ define a function $f$ on the infinite dimensional space $L^2(\Sigma,\e^h\,d\Vol)$ as follows:  If $w\in L^2(\Sigma,\e^h\,d\Vol)$, then
\begin{equation}
f(w)=\sum_i\frac{\mu_i}{2}\,\langle w,w_i\rangle^2\, .
\end{equation}
As in the finite dimensional case, the negative gradient flow of $f$ is:
\begin{equation}
\Psi_t(w)=\e^{-\mu_i t}\langle w_i,w\rangle\, .
\end{equation}

Of particular interest is when $\Sigma^n\subset \RR^{n+1}$ is a self-shrinker, $h(x)=-\frac{|x|^2}{4}$, and the basis $w_i$ are eigenfunctions with eigenvalues $\mu_i$ of a self-adjoint operator $L$ of the form
\begin{equation}
L\, w=\Delta\,w+|A|^2\,w-\frac{1}{2}\langle x,\nabla w\rangle+\frac{1}{2}w\, .
\end{equation}
The reason this is of particular interest is because in \cite{CM1} it was shown that the Hessian of the $F$-functional is given by
\begin{equation}
\Hess_F(v,w)=- \left( 4 \, \pi \right)^{ - \frac{n}{2} } \, \int_{\Sigma} v\,L\,w\,\e^{-\frac{|x|^2}{4}}\, .
\end{equation}

For an $f$ of this form,  the negative gradient flow is equal to the heat flow of the linear heat operator $(\partial_t-L)$.  Moreover, this linear heat flow is the linearization of the rescaled MCF at the self-shrinker.   It follows that the rescaled MCF near the self-shrinker is approximated by the negative gradient flow of $f$.  This same fact is also reflected by fact that if we formally write down the first three terms in the Taylor expansion of $F$, then we get the value of $F$ at $\Sigma$ plus a first order polynomial which is zero since $\Sigma$ is a critical point of $F$ plus a polynomial of degree two which is given by the Hessian of $F$ and is exactly $f$.   This gives a heuristic explanation for the above theorem:  The dynamics of the negative gradient flow of the $F$ functional should be well approximated by the dynamics for its second order Taylor polynomial.   

\section{Rigidity of cylinders and space-time neighborhoods of generic singularities for MCF}

In this section we discuss some joint work in progress between the first two authors and Tom Ilmanen; see \cite{CIM} for more details.

The aim of this work is two-fold.  First, to show that round generalized cylinders are rigid in a very strong sense.  Namely, any other self-shrinker that is sufficiently close to one of them on a large, but compact set, and with a fixed but arbitrary entropy bound must indeed itself be a round generalized cylinder.    

The second aim is to show a canonical neighborhood theorem near any generic singularity for general mean curvature flow.   This would assert that if at a singularity a tangent flow is a generalized cylinder, then in a space-time neighborhood of the singularity the flow has positive mean curvature.   Thus, for  any generic mean curvature flow,   the flow looks like one  of the very special flows, namely those that have positive mean curvature, in a neighborhood of each singularity.  Incidentally, as mentioned earlier, when an entire time-slice has positive mean curvature then so has any later time-slice.  For rotationally symmetric hypersurfaces this property of mean convexity near singularities was shown by ODE techniques by Altschuler, Angenent, and Giga where they called it the ``attracting axis theorem''.   

\vskip2mm
More precisely the first aim of the work \cite{CIM} is to show the following:

\begin{Con}		\label{c:uniq1}
Given $n$, $\alpha>1$, there exists $R=R(n, \alpha) > 0$ so that if $\Sigma^n \subset \RR^{n+1}$ is a smooth embedded self-shrinker with entropy $\leq \alpha$ and 
\begin{itemize}
\item $H \geq 0$ on $B_R \cap \Sigma$, 
\end{itemize}
then $H \geq 0$ on $\Sigma$ and, thus, $\Sigma$ is a generalized cylinder $\SS^k \times \RR^{n-k}$.
\end{Con}

The second aim of the work \cite{CIM} is the following closely related canonical neighborhood statement:

\begin{Con}		\label{c:uniq2}
Suppose that $M_t$ is a MCF flow of smooth closed hypersurfaces in $\RR^{n+1}$.  If the flow has a cylindrical singularity at time $t_0$ and at the point $x_0$ in space $\RR^{n+1}$, then in an entire space-time neighborhood of $(x_0,t_0)$ the evolving hypersurfaces has positive mean curvature.
\end{Con}

Both of these conjectures are shown to be true in \cite{CIM} with some mild extra assumptions and a proof of the full conjectures seems within reach.


\end{document}

%% file: cmp8.tex
\section{Surgery in low dimensions}
The well known surgery exact sequence of Browder-Novikov-Sullivan-Wall
\cite{w2} breaks down in dimensions below 5. In this section we discuss what remains
in low dimensions, and some of the implications this has in higher
dimensions.

Surgery deals with existence and uniqueness of manifold structures on a
given Poincar\'e Duality space. We shall not discuss existence, since this
is precisely what breaks down in low dimensions, and if a Poincar\'e Duality space is homotopy equivalent to a
manifold, we might as well replace the Poincar\'e Duality space by a
manifold.

The first observation we make is that the terms in the surgery exact
sequence are defined in all dimensions.

\begin{Def} Let $M$ be a compact manifold without boundary. An
element in the structure set is a homotopy equivalence
\[
M_1\simeq M
\]
of manifolds. Two such elements are equivalent if there is a homotopy
commutative diagram of manifolds
\[
\xymatrix{
M_1\ar@{_{(}->}[d]\ar[dr]\\
W\ar[r]&M\\
M_2\ar@{^{(}->}[u]\ar[ur]
}
\]

\noindent where the inclusions of $M_i\subset W$ are homotopy equivalences,
and $\partial W$ is the disjoint union of $M_1$ and $M_2$. Such a $W$ is
called an $h$-cobordism between $M_1$ and $M_2$.
\end{Def}

 The structure
set is usually denoted $\cs_h(M)$, and it comes in smooth, $\pl$, and
topological versions. It can also be varied by requiring the homotopy
equivalences to be simple, in this note however, we shall stick to just
homotopy equivalences.

There is a relative version of the structure set where we allow $M$ to have
a boundary and require $h:M_1\to M$ to be a homeomorphism
($\pl$-homeomorphism, diffeomorphism) on the boundary. This gives rise to the
higher structure sets $\cs(M\times D^n \text{ rel. } \partial)$. In this note
we shall not be considering further relativisation (e.~g. allowing homotopy
equivalence on a part of the boundary).

The next term in the surgery exact sequence is the normal invariant. There 
are many mistakes in the literature concerning basepoints, so to be precise
we shall consider the normal invariant to be the set of homotopy classes of based maps from
$M_+$ ($M$ with a disjoint basepoint) to $G/O$ in the smooth case, $G/\pl$
in the $\pl$ case, and $G/\Top$ in the topological case. In case $M$ has a
boundary the normal invariant is $[M/\partial M, G/O]_*$, so if $M$ is
closed $[M\times D^n/\partial, G/O]_*$ is the same as
$[\Sigma^n(M_+),G/O]_*$ which is the reason to prefer $[M_+,G/O]_*$ to
the free homotopy classes $[M,G/O]$, even though it is the same.

There is a map $\cs(M)\to [M_+,G/O]_*$
and similarly for the $\Top$ and $\pl$ case defined as follows:
Given a homotopy equivalence $M_1\simeq M$ of smooth
manifolds, the uniqueness of the Spivak normal fibration \cite{spi1} produces a
commutative diagram 
\[
\xymatrix{
M_1\ar[d]\ar[dr]\\
M\ar[r]&BG
}
\]

The smooth structures on $M_1$ and $M$ produce lifts to $BO$ and since
$BO$ and $BG$ are loop spaces this produces a map to $G/O$, the fibre. In
case $M$ has a boundary, these lifts agree on the boundary, so we get a
basepoint preserving map $M/\partial M\to G/O$.

The last term in the surgery exact sequence is the Wall group. To indicate
this consider a degree one normal map
\[
\xymatrix{
\nu_{M_1}\ar[d]\ar[r]&\xi\ar[d]\\
M_1\ar[r]&M
}
\]

\noindent i.~e. a map sending the fundamental class of $M_1$ to the fundamental class
of $M$, covered by a map of bundles from the normal bundle of $M_1$ to some
bundle over $M$ (vector bundle, $\Top$-bundle or $\pl$-bundle) The problem of
surgery is to produce a bordism $W$ from $M_1$ to $M_2$ covered by a map of
bundles, such that $M_2\to M$ is a homotopy equivalence. Classically this is
done by making the map $M_1\to M$ highly connected by surgery, i.~e.
whenever there is an obstruction
\[
\xymatrix{
\SS^i\ar[r]\ar[d]&D^{i+1}\ar[d]\\
M_1\ar[r]&M
}
\]

\noindent to the map being a homotopy equivalence , $\SS^i\to M$ is replaced by an
embedding $\SS^i\times D^{n-i}\to M$ (which exists when $i$ is small, the
extension of $\SS^i\to M$ to $\SS^i\times D^{n-i}\to M$ is assured by the
bundle information), and then replace $\SS^i\times D^{n-i}$ by $D^{i+1}\times
\SS^{n-i-1}$. Actually we need the bordism covered by bundle maps, but the bordism is obtained
by gluing $M\times I$ and $D^{i+1}\times D^{n-i}$ along $\SS^i\times D^{n-i}\subset M\times 1$ and 
$\SS^i\times D^{n-i}\subset \partial(D^{i+1}\times D^{n-i})$. This process breaks down when we reach the middle dimension,
and the obstruction is not understood in low dimensions. Using
Ranickis algebraic surgery point of view \cite{ra4} however, we can immediately pass
to algebra, and we then obtain an algebraically defined obstruction in
an algebraically defined group, $L_n^h(\ZZ\pi)$, $\pi=\pi_1(M)$, which is an obstruction in all dimensions, but maybe not
the whole obstruction in low dimensions.

There is a map
 \[
[M_+,G/O]_*\to L_n(\ZZ\pi)
\]
\noindent defined as follows: Embed the smooth manifold in $\RR^k$, $k$ large, let $D(M)$
be the normal disk bundle, $S(M)$ the normal sphere bundle. The Thom map is
the map from $\SS^k$ thought of as $\RR^k$ with an extra point at infinity to
the Thom space $D(M)/S(M)$ sending everything outside $D(M)$ to the
collapsed point. A map $M\to G/O$ gives rise to another bundle $\xi$ over $M$ together with a fibre homotopy equivalence 
of the corresponding sphere bundles. Consider the composite $\SS^k \to D(M)/S(M)\to D(\xi)/S(\xi)$. Making this map transverse to the
$0$-section in $\xi$ produces a manifold $M_1$, and a bundle map from the normal bundle of $M_1$ to $\xi$. The fundamental class of
$M_1$ is sent to the fundamental class of $M$ by a Thom isomorphism argument, so we have a surgery problem, and we now pass to 
algebra to produce an element in $L^h_n(\ZZ\pi)$.

The bottom part of the surgery exact sequence 
\[
\cs_h(M)\to [M_+,G/O]_*\to L_n^h(\ZZ\pi)
\]

\noindent is now established, and similarly for the case with boundary. The composite is obviously the
zero map since starting with a homotopy equivalence there is no obstruction
to obtaining a homotopy equivalence. In low dimensions however, it may not be
true that we can perform the surgery to obtain a homotopy equivalent
manifold, even though the algebraically defined
surgery obstruction vanishes. In dimensions at least $5$ the necessary
embedding theorems are available so we have exactness in the sense of
pointed sets. There is no group structure on the structure set, but in the
case the manifold is of the form $N\times I$, there is an obvious associative monoid
structure on the structure set and a group structure on the other terms, and the maps are indeed homomorphisms. In dimensions at least
$5$ this monoid structure is easily seen to be a group structure

The sequence extends to the left as follows: Given a closed smooth manifold
$M^n$ with fundamental group $\pi$ and an element in $L_{n+1}(\ZZ\pi)$ one
may attempt to produce a bordism
\[W\to M\times I
\]
with $\partial W$ the disjoint union of $M$ and $M_1$, covered by bundle
maps, realizing the given surgery obstruction, with $M\to M$ the
identity and $M_1\to M$ a homotopy equivalence. The idea is now to send the given element to this homotopy equivalence thought of as
an element in the structure set. This is always possible
when the dimension of $M$ is at least $5$. In low dimensions it may or
may not be possible, and it is not clear we get a well defined map 
$L_{n+1}^h(\ZZ\pi)\to\cs_h(M)$ either in low dimensions. So in general we
only get a partially defined, maybe not well defined action of
$L_{n+1}(\ZZ\pi)$ on $\cs(M)$.

Given an element in $[\Sigma(M_+),G/O]_*$, a transversality argument as
above produces a bordism $W$ and a map $W\to M\times I$ covered by bundle
maps with $\partial W$ the disjoint union of $M$ and $M$ and the map on the
boundary the identity of $M$ to $M\times 0$, $M\times 1$ respectively. We
thus get
\[
\xymatrix{
[\Sigma M_+,G/O]\ar[r]& L^h_{n+1}(\ZZ\pi)\ar@{-->}[r]& \cs(M)
}
\]

\noindent The second map is defined on the image of the first map, and the composite
is obviously trivial.

It should now be clear how to extend the sequence to the left, and since it
will be a manifold with an $I$-factor it will be groups and homomorphisms to
the left

The surgery groups of the trivial group were essentially calculated by
Kervaire and Milnor to be $\ZZ$ for $n\equiv 0 (4)$, $\ZZ/2$ for $n\equiv 2
(4)$ and $0$ otherwise. The surgery obstruction for $n\equiv 0 (4)$ is given
as follows: Do surgery below the middle dimension (also possible in dimension
$4$), the kernel in homology is now an even, symmetric nonsingular matrix over $\ZZ$. Such a matrix has index 
divisable by $8$, and the obstruction is the index divided by $8$. This can
also be calculated as the difference of the indexes of the
manifolds divided by 8.

Since every $\pl$-manifold in dimension less than $7$ is smoothable, Rohlins theorem states that if $M^4$ is a smooth or $\pl$ $4$-manifold
with $w_1=w_2=0$, then the index of $M$ is divisable by 16. Index=16 however, is realizable using the Kummer surface, and using
connected sum we may realize every multiple of 16.

\section{The surgery exact sequence for a disk}

Now consider the $\pl$ disk. The surgery exact sequence is
\[
\pi_{n+1}(G/\pl)\to L_{n+1}(\ZZ)\to \cs(D^n \text{ rel. } \partial) \to \pi_n(G/\pl)\to L_n(\ZZ)
\]

when $n\ge 5$ consider an element in $\cs(D^n \text{ rel. } \partial)$
\[
\xymatrix{
\SS^{n-1}\ar[d]\ar[r]^{id}&\SS^{n-1}\ar[d]\\
W\ar[r]&D^n
}
\]

Cutting out a small open disc in the interior of $W$ produces an $h$-cobordism hence a product, so $W$ is a disk.
Using a cone construction we get a $\pl$-homeomorphism to $D^n$ homotopic to the original map relative to the boundary. This latter part
is of course what does not work in the smooth case. This shows $\pi_n(G/\pl)\cong L_n(\ZZ)$ for $n\ge 5$. For $n<5$, Kervaire and Milnor
use that $\pl/O$ is $6$-connected and the $J$-homomorphism $O\to G$ is known in low dimensions, showing that 
\[
\pi_n(G/\pl)=\begin{cases} \ZZ\qquad n\equiv 0 (4)\\\ZZ/2\qquad n\equiv 2 (4)\\0\qquad n\equiv 1 (2)
\end{cases}
\]

\noindent Let us analyze the picture in low dimensions
\[
\pi_5(G/O)\to L_5(\ZZ)\to \cs(D^4 \text{ rel. } \partial) \to \pi_4(G/O)\to L_4(\ZZ).
\]

\noindent $L_5(\ZZ)=\pi_5(G/O)=0$ and $\pi_4(G/O)=L_4(\ZZ)=\ZZ$, but the map by Rohlins theorem is multiplication by 2.

\begin{Obs} Surgery theory suggests $\cs(D^4 \text{ rel. } \partial)=0$.
\end{Obs}
\noindent This is actually true as we shall see (in the equivalent case of $\SS^4$) in the next section.

Next consider $n=3$
\[
\pi_4(G/O)\to L_4(\ZZ)\to \cs(D^3 \text{ rel. }\partial) \to \pi_3(G/O)\to L_3(\ZZ)
\]

We have $\pi_3(G/O)=L_3(\ZZ)=0$, $\pi_4(G/O)=L_4(\ZZ)=\ZZ$, but the same argument as before says the map is multiplication by $2$, so
we get
 
\begin{Obs} Surgery theory suggests $\cs(D^3 \text{ rel. } \partial)=\ZZ/2$.
\end{Obs}
\noindent This is of course not true by Perelmans work. 

In dimension $2$ we have 
\[
\pi_3(G/O)\to L_3(\ZZ)\to \cs(D^2 \text{ rel. } \partial) \to \pi_2(G/O)\to L_2(\ZZ)
\]

\noindent and the map $\pi_2(G/O)\to L_2(\ZZ)$ is an isomorphism since the Kervaire invariant element is realized by $T^2\to S^2$ covered by
bundles defined by the left invariant framing of $T^2$. 

\begin{Obs} Surgery theory suggests $\cs(D^2 \text{ rel. }
\partial)=0$.
\end{Obs}

This is of course correct. 

In dimension $1$ we have 
\[
\pi_2(G/O)\to L_2(\ZZ)\to \cs(D^1 \text{ rel. } \partial) \to \pi_1(G/O)\to L_1(\ZZ)
\]
\noindent and the map $\pi_2(G/O)\to L_2(\ZZ)$ again is an isomorphism for the same reason as above. 

\begin{Obs} Surgery theory suggests $\cs(D^1 \text{ rel. } \partial)=0$.
\end{Obs}

\noindent This is of course correct.

Finally in dimension $0$ we have 
\[
\pi_1(G/O)\to L_1(\ZZ)\to \cs(D^0 \text{ rel. } \partial) \to \pi_0(G/O)\to L_0(\ZZ)
\]
\noindent Leading to the obvious
\begin{Obs} Surgery theory suggests $\cs(D^0)=0$.
\end{Obs}

These low dimensional phenomena do not make sense since surgery does not work in low dimensions, but they do have consequences in
higher dimensions and lead to relatively simple proofs of topological invariance of Pontrjagin classes, and the theory of topological
manifolds as developed by Kirby and Siebenmann.

\section{Dimension 4}

Let $M^4$ be a closed smooth $4$-dimensional manifold 

\begin{Thm}
Assume $[\Sigma(M_+), G/O]_*\to L_5(\ZZ\pi)$ is an epimorphism and $[M_+,G/O]_*\to L_4(\ZZ\pi)$ is a monomorphism, then 
$\cs(M)=0$.
\end{Thm}
\begin{proof}
An element in $\cs(M)$ is a homotopy equivalence $M_1\to M$. The composite map to $L_4(\ZZ\pi)$ is always zero so $\cs(M)\to
[M_+,G/O]_*$ must be the zero map. This means there is a normal cobordism $W$ with boundary the disjoint union of $M_1$ and $M$, and a
degree one normal map $W\to M\times I$. This normal map has a surgery obstruction $\sigma\in L_5(\ZZ\pi)$. Now choose an element in 
$[\Sigma(M_+),G/O]_*$ which hits $-\sigma$. This produces a normal cobordism $W_1\to M\times I$ which is the identity on the boundary.
Gluing $W$ and $W_1$ together the surgery obstruction is $0$, but we are in dimension $5$, so we may perform the surgery to obtain an
$h$-cobordism between $M_1$ and $M$ and a homotopy equivalence to $M\times I$.
\end{proof}

\begin{Cor}
Assume $M$ is a closed $4$-dimensional manifold homotopy equivalent to $\SS^4$. Then $M$ is $h$-cobordant to $\SS^4$
\end{Cor}
\begin{proof}
$L_5(\ZZ)=0$ so the first condition in the theorem above is obviously satisfied. $G/O$ is connected and simply connected, so 
$[{\SS^4}_+,G/O]_* = \pi_4(G/O)=\ZZ$ and the map to $L_4(\ZZ)=\ZZ$ is multiplication by $2$ by Rohlins theorem as above. This means
$\cs(\SS^4)=0$ hence $M$ is $h$-cobordant to $\SS^4$.
\end{proof}

\begin{proof}[Proof of Theorem \ref{four}] We want to show that $M$ can be smoothly embedded in $\RR^5$. 
By the corollary above there is an $h$-cobordism $W$ between $M$ and $\SS^4$. Now $W\times \SS^1$ is a
$6$-dimensional $h$-cobordism and $\wh(\ZZ)=0$, so $M\times \SS^1$ is diffeomorphic to $\SS^4\times \SS^1$. Hence the universal covers
$M\times \RR$ and $\SS^4\times \RR$ are diffeomorphic, but $\SS^4\times \RR$ is diffeomorphic to $\RR^5\setminus 0$.
\end{proof}

\section{Bounded surgery}
The material in this section is taken from \cite{fp2}.
\begin{Def} Let $X$ be a metric space, $M_1$ and $M$ topological spaces, $p:M\to X$ a proper map. Then a map
$f:M_1\to M$ is said to be a bounded homotopy equivalence if there is a map $g:M\to M_1$ and homotopies $G$ from $g\circ f$ to the
identity and $H$ from $f\circ g$ to the identity such that $p(H(z\times I))$ and $p(f(K(y\times I)))$ have uniformly bounded diameter.
\end{Def}

\begin{Rem} This concept of course is only interesting when $M_1$ and $M$ are non-compact
\end{Rem}

In \cite{fp2} a bounded surgery theory was developed classifying manifolds up to bounded homotopy equivalence for a large class of
metric spaces $X$, e.~g. $X=\RR^n$, and a surgery exact sequence was established in dimensions $\ge 5$. The normal invariant term is
the same as in the compact case, but the $L$-group term was defined using certain additive categories that give an algebraic criterion
for bounded homotopy equivalence similarly to the Whitehead theorem in the classical case. This additive category is denoted
$\cc_X(\ZZ)$ in the simply connected case. We thus have an exact sequence:

\[
\ldots \to [\Sigma(M_+),G/O]_*\to L_{n+1}(\cc_X(\ZZ))\to \cs_b(M\to X)\to [M_+,G/O]_*\to L_n(\cc_X(\ZZ))
\]

A specially interesting case is $X=\RR^n$ and $M=D^k\times \RR^n$ with $p$ projection on the second factor. In this case our
observations above lead to the following theorem:

\begin{Thm} When $n+k\ge 5$, the bounded structure set $\cs(D^k\times\RR^n\to \RR^n \text{ rel. } \partial)$ is $0$ when $k\ne 3$,
and $\ZZ/2$ when $n=3$.
\end{Thm}
\begin{proof} Crossing with $\RR^n$ is obviously an isomorphism on the normal invariant term, since this term is homotopy theoretic,
and $\RR^n$ is contractible. One may prove algebraically \cite{ra6} that crossing with $\RR^n$ induces an isomorphism $L_k(\ZZ)\to
L_{n+k}(\cc_{\RR^n}(\ZZ))$, but then the observations above calculate the maps in the bounded surgery sequence which is exact when
$n+k\ge 5$.
\end{proof}

To study the difference between $\pl$ and $\Top$ manifolds, one considers $\pi_k(\widetilde \Top(n), \widetilde \pl(n))$. An element is given by a
homeomorphism $\Delta(k)\times\RR^n\to \Delta(k)\times\RR^n$, where $\Delta(k)$ is a standard $k$-simplex, restricting to a
$\pl$-homeomorphism $\sigma\times\RR^n\to\sigma\times\RR^n$ for each face $\sigma$ of $\Delta(k)$. Such an element is $0$ if there is
an isotopy, fixing the boundary, to a $\pl$-homeomorphism. If all such elements were $0$ a relatively simple argument due to Kirby,
would produce a $\pl$-structure on a topological manifold. In this direction we get the following 

\begin{Thm} When $n+k\ge 5$ there is a monomorphism $\pi_k(\widetilde \Top(n), \widetilde \pl(n)) \to \cs_b(D^k\times \RR^n\to \RR^n)$
\end{Thm}

Since we have calculated the target this means that $\pi_k(\widetilde \Top, \widetilde \pl)$ is $0$ for $k\ne 3$ and for $k=3$ it is
either $0$ or $\ZZ/2$. Notice this implies topological invariance of Pontrjagin classes since the stable maps $BO\to B\pl$ and $B\pl\to B\Top$ are
then both rational homotopy equivalences.

\begin{proof} To define the map consider a homeomorphism $\Delta(k)\times\RR^n\to \Delta(k)\times\RR^n$ which restricts to a 
$\pl$ homeomorphism on each simplex in the boundary crossed with $\RR^k$ hence to a $\pl$-homeomorphism on the boundary. 
Identifying $D^k$ and $\Delta(k)$ a homeomorphism is obviously a bounded homotopy equivalence, bounded by $0$, so it defines an element
in the structure set (which is a group in this case). To see this map is monic, we change the metric on $D^k\times \RR^n$ such that
$D^k\times\{x\}$ has the standard metric on $D^k$ multiplied by $||x||+1$. Let us denote this metric space by $X$. We now have a
proper Lipschitz map $X\to \RR^n$, and it turns out that induces an isomorphism of structure sets $\cs(D^k\times\RR^n\to X)\to
\cs(D^k\times\RR^n\to \RR^n)$. This is a five lemma argument. On the normal invariant term the metric space plays no role, and on the
$L$-group term an algebraic argument provides the necessary argument, see \cite{fp2}. 
Now consider an element $h$ going to $0$ in $\cs(D^k\times\RR^n \to \RR^n)$. The above mentioned map obviously factors through
$\cs(D^k\times\RR^n\to X)$ since a homeomorphism is a bounded homotopy equivalence no matter what metric space it is measured in. This
means there is a bounded isotopy relative to the boundary of $h$ to a $\pl$-homeomorphism $g$ which is a uniformly bounded distance
from the identity when measured in $X$. Now consider $h\circ g^{-1}:D^k\times\RR^n \to D^k\times\RR^n$. This is the identity on the
boundary. Thinking of $D^k\times\RR^n$ as  an open subset of $D^k\times D^n$ in a radial way, one sees that this homeomorphism may be
completed to a homeomorphism of $D^k\times D^n$ by the identity because boundedness measured in $X$ translates to smallness near the
boundary of $D^n$. An Alexander isotopy now provides an isotopy relative to the boundary of $h\circ g^{-1}$ to the identity, so $h$ 
is isotopic to $g$.
\end{proof}

It can be proved that $\pi_3(\widetilde\Top(n),\widetilde\pl(n))=\ZZ/2$ when $n+k\ge 6$, in other words the above map is onto, 
but this requires a proof that a certain bounded homotopy equivalence is bounded homotopic to a homeomorphism using Quinns end theorem,
 and is probably not easier than Kirby-Siebenmanns original argument.